\documentclass{amsart}
\usepackage{amsmath,amssymb}

\newtheorem*{cuhcrit}{Kreisel's Criterion}
\newtheorem*{cuhhypo}{Computable Universe Hypothesis}
\newtheorem{cuhdef}{Definition}[section]
\newtheorem{cuhmodel}[cuhdef]{Model}
\newtheorem{cuhtheorem}[cuhdef]{Theorem}

\newcommand{\cuhinterval}[2]{(#1\,;#2)}
\newcommand{\cuhbinterval}[2]{\bigl(#1\,;#2\bigr)}
\newcommand{\cuhBinterval}[2]{\Bigl(#1\,;#2\Bigr)}
\newcommand{\cuhomicron}{o}
\DeclareMathOperator{\cuhdom}{dom}
\DeclareMathOperator{\cuhsgn}{sgn}

\providecommand{\bysame}{\leavevmode\hbox to3em{\hrulefill}\thinspace}

\begin{document}

\title{The Computable Universe Hypothesis}

\author{Matthew~P. Szudzik}

\date{28 January 2012}

\begin{abstract}
When can a model of a physical system be regarded as computable?  We provide the definition of a \emph{computable physical model} to answer this question.  The connection between our definition and Kreisel's notion of a mechanistic theory is discussed, and several examples of computable physical models are given, including models which feature discrete motion, a model which features non-discrete continuous motion, and probabilistic models such as radioactive decay.  We show how computable physical models on effective topological spaces can be formulated using the theory of type-two effectivity (TTE).  Various common operations on computable physical models are described, such as the operation of coarse-graining and the formation of statistical ensembles.  The definition of a computable physical model also allows for a precise formalization of the \emph{computable universe hypothesis}---the claim that all the laws of physics are computable.
\end{abstract}
\maketitle

\section{Introduction}

A common way to formalize the concept of a \emph{physical\index{physical models} model} is to identify the states of the system being modeled with the members of some set $S$, and to identify each observable quantity of the system with a function from $S$ to the real numbers.\footnote{A more detailed account of this formalism is available in reference~\cite{rR78}.}
For example, a simple model of planetary motion, with the Earth moving in a circular orbit and traveling at a uniform speed, is the following.
\begin{cuhmodel}[Simple Planetary Motion] \label{m:real}
Let $S$ be the set of all pairs of real numbers $(t,a)$ such that $a=360\bigl(t-\lfloor t\rfloor\bigr)$, where $\lfloor t\rfloor$ denotes the largest integer less than or equal to $t$.  The angular position of the Earth, measured in degrees, is given by the function $\alpha(t,a)=a$.  The time, measured in years, is given by the function $\tau(t,a)=t$.
\end{cuhmodel}
\noindent
If we wish, for example, to compute the position of the Earth after $2.25$ years, we ask: ``For which states $(t,a)$ does $\tau(t,a)=2.25$?''  There is only one such state, namely $(2.25,90)$.  Therefore, the position of the Earth after $2.25$ years is $\alpha(2.25,90)=90$ degrees.  We say that the model is \emph{faithful}\index{physical models!faithful} if and only if the values of the observable quantities in the model match the values that are physically observed.

Church\index{Church, Alonzo} and Turing\index{Turing, Alan} hypothesized that the functions which are effectively computable by humans are exactly the recursive functions.\footnote{Readers unfamiliar with the definition of a recursive function or related terminology, such as uniformity, should consult reference~\cite{hR67}.  The original justifications for identifying the effectively computable functions with the recursive functions can be found in references~\cite{aC36,aT36,aT37b}.}
There have been several attempts~\cite{rR62,kZ67,eF90,sW02,mT08} to extend the Church-Turing\index{Church-Turing thesis} thesis to physics, hypothesizing that the laws of physics are, in some sense, computable.  But given an arbitrary physical model, it has not been clear exactly how one determines whether or not that model is to be regarded as computable.  To date, the best attempt at providing such a definition has been Kreisel's\index{Kreisel, Georg} notion of a mechanistic\index{mechanistic theories} theory~\cite{gK74}.  Kreisel\index{Kreisel, Georg} suggested the following.
\begin{cuhcrit}
\index{Kreisel's criterion}The predictions of a physical model are to be regarded as computable if and only if every real number which is observable according to the model is recursive relative to the data uniformly.
\end{cuhcrit}
\noindent
But many seemingly innocuous models have failed to satisfy Kreisel's criterion.  For example, the simple model of planetary motion (Model~\ref{m:real}) fails because given a real number representing the time $t$, there is no effectively computable procedure which determines the corresponding angle $a$ when $a$ is near the discontinuity at $360$ degrees.  Models which intuitively seem to have computable predictions often fail to satisfy Kreisel's criterion because discontinuities in their formalisms prevent the models' predictions from being regarded as computable, despite the fact that there are no discontinuities in the actual physical phenomena being modeled~\cite{wM95}.

Rather than using Kreisel's criterion to prove that the predictions of established models are computable, an alternate approach is to supply a restrictive formalism which guarantees that the predictions of models expressible in that formalism are computable.  This has been the approach taken in references~\cite{rR62,kZ67,eF90,sW02}.  But difficulties have been encountered expressing important established models in these formalisms.  For example, Rosen\index{Rosen, Robert}~\cite{rR62} was unable to describe radioactive decay in the formalism that he had proposed, and work is ongoing to describe established  physical models in other computable formalisms.

It is the goal of this paper to provide a general formalism for describing physical models whose predictions are computable, and to show that the computable formalisms studied by previous authors are special cases of our general formalism.  In particular, we show in Section~\ref{s:kreisel} that among the members of a large class of physical models, each physical model satisfying Kreisel's criterion has a corresponding model in our formalism.  We also avoid some of the difficulties which, for example, prevented the simple model of planetary motion (Model~\ref{m:real}) from being regarded as computable, as will be seen in Section~\ref{s:continuous}.  Our approach also avoids the difficulty that Rosen\index{Rosen, Robert} encountered with radioactive decay, as will be seen in Section~\ref{s:radio}.

\section{Computable Physical Models} \label{s:cuh}

The central problem is that physical models use real numbers to represent the values of observable quantities, but that recursive functions are functions of non-negative integers, not functions of real numbers.  To show that a model is computable, the model must somehow be expressed using recursive functions.  Careful consideration of this problem, however, reveals that the real numbers are not actually necessary in physical models.  Non-negative integers suffice for the representation of observable quantities because numbers measured in laboratory experiments necessarily have only finitely many digits of precision.  For example, measurements of distances with a measuring stick will always be non-negative integer multiples of the smallest division on the measuring stick.  So, we suffer no loss of generality by restricting the values of all observable quantities to be expressed as non-negative integers---the restriction only forces us to make the methods of error analysis, which were tacitly assumed when dealing with real numbers, an explicit part of each model.

Non-negative integers are not only sufficient for the description of direct physical measurements, but are also sufficient for encoding more complex data structures---allowing us to define recursive functions on those data structures.  For example, a pair of two non-negative integers $x$ and $y$ can be encoded as a single non-negative integer $\langle x,y\rangle$ using Cantor's\index{Cantor, Georg} pairing function
\begin{equation*}
\langle x,y\rangle=\frac{1}{2}(x^2+2x y+y^2+3x+y)
\end{equation*}
A pair $\langle x,y\rangle$ of non-negative integers will also be called a \emph{length two sequence} of non-negative integers. A triple (or equivalently, length three sequence) of non-negative integers $x$, $y$, and $z$ can be encoded as $\langle\langle x,y\rangle,z\rangle$, and so on.  We write $\langle x,y,z\rangle$ as an abbreviation for $\langle\langle x,y\rangle,z\rangle$.  An integer $i$ can be encoded as a non-negative integer $\zeta(i)$ using the formula
\begin{equation*}
\zeta(i)=
\begin{cases}
-2i-1 &\text{if $i<0$}\\
2i &\text{if $i\geq 0$}
\end{cases}
\end{equation*}
And a rational number $\frac{a}{b}$ in lowest-terms with $b>0$ can be encoded as a non-negative integer $\rho(\frac{a}{b})$ using the formula
\begin{equation*}
\rho\bigl(\frac{a}{b}\bigr)=\zeta\bigl((\cuhsgn a)2^{\zeta(a_1-b_1)} 3^{\zeta(a_2-b_2)} 5^{\zeta(a_3-b_3)} 7^{\zeta(a_4-b_4)} 11^{\zeta(a_5-b_5)}\cdots\bigr)
\end{equation*}
where $a=(\cuhsgn a)2^{a_1} 3^{a_2} 5^{a_3} 7^{a_4} 11^{a_5}\cdots$ is the prime factorization of the integer $a$, and similarly for $b$.  We write $\cuhinterval{q}{r}$ as an abbreviation for the pair of rational numbers $\bigl\langle\rho(q),\rho(r)\bigr\rangle$.

Historically, authors who have wished to restrict themselves to physical models whose predictions are computable have chosen from among a handful of formalisms.  For example, Zuse\index{Zuse, Konrad}~\cite{kZ67} and Fredkin\index{Fredkin, Edward}~\cite{eF90} have formalized their models as cellular automata, with each cell of an automaton representing a discrete unit of space and each step of computation in the automaton representing a discrete unit of time.  Wolfram\index{Wolfram, Stephen}~\cite{sW02} has formalized his models in a variety of computational systems, including cellular automata, but has favored network systems for a model of fundamental physics.  In each of these cases, the states of a physical system are represented by the states of a computational system (for example, a cellular automaton or a network system) which can be encoded as non-negative integers using the techniques just described.  The resulting set of non-negative integers is a recursive set, and the observable quantities of the system are recursive functions of the members of that set. This immediately suggests the following definition.
\begin{cuhdef}
A \emph{computable\index{physical models!computable} physical model} of a system is a recursive set $S$ of states with a total recursive function $\phi$ for each observable quantity of the system.  $\phi(s)$ is the value of that observable quantity when the system is in state $s$.
\end{cuhdef}

So, in a computable physical model the set $S$ is a set of non-negative integers, and each observable quantity is a function from non-negative integers to non-negative integers.  The models considered by Zuse\index{Zuse, Konrad}, Fredkin\index{Fredkin, Edward}, and Wolfram\index{Wolfram, Stephen} are necessarily special sorts of computable physical models, and the set of all computable physical models is a proper subset of all physical models.  In order to avoid all ambiguity, we insist that observable quantities be defined operationally~\cite{pB27} in computable physical models, so that, for example, if there were an observable quantity corresponding to time, then that observable quantity would be the time as measured with a specific conventionally-chosen clock in a specific conventionally-chosen reference frame.

An immediate consequence of the definition of a computable physical model is that we can give a precise formal counterpart to the informal claim that all the laws of physics are computable.
\begin{cuhhypo}
\index{computable universe hypothesis}The universe has a recursive set of states $U$.  For each observable quantity, there is a total recursive function $\phi$.  $\phi(s)$ is the value of that observable quantity when the universe is in state $s$.
\end{cuhhypo}

By a \emph{distinguishable\index{distinguishable systems} system}, we mean any system for which there is an observable quantity $\phi$ such that $\phi(s)=1$ when the system exists in the universe, and such that $\phi(s)=0$ otherwise.  For example, if the system being studied is the orbit of the Earth, then $\phi(s)=0$ when state $s$ corresponds to a time before the formation of the Earth, and $\phi(s)=1$ when the Earth exists and is orbiting the Sun.  Note that the set of states $s$ in $U$ for which $\phi(s)=1$ is itself a recursive set whenever $U$ and $\phi$ are recursive.  So, the computable universe hypothesis implies that computable physical models are sufficient for modeling any distinguishable system in the universe---the set of states of that distinguishable system is the set of all members $s$ of $U$ for which $\phi(s)=1$, and the observable quantities of the distinguishable system are necessarily a subset of the observable quantities of the universe.

\section{Discrete Planetary Motion}

As a first example of a computable physical model, consider the following model of planetary motion.
\begin{cuhmodel}[Discrete Planetary Motion] \label{m:discrete}
Let $S$ be the set of all pairs $\bigl\langle\cuhinterval{r}{s},\cuhinterval{p}{q}\bigr\rangle$ such that
\begin{align*}
r&=\frac{i}{10}-\frac{1}{100} & p&=360\bigl(r-\lfloor r\rfloor\bigr)\\
s&=\frac{i+1}{10}+\frac{1}{100} & q&=360\bigl(s-\lfloor s\rfloor\bigr)
\end{align*}
for some integer $i$ between $-20000$ and $20000$.  The angular position of the Earth, represented as a range of angles measured in degrees, is given by the function $\alpha\bigl\langle\cuhinterval{r}{s},\cuhinterval{p}{q}\bigr\rangle=\cuhinterval{p}{q}$.  The time interval, measured in years, is given by the function $\tau\bigl\langle\cuhinterval{r}{s},\cuhinterval{p}{q}\bigr\rangle=\cuhinterval{r}{s}$.
\end{cuhmodel}

This is a discrete model.  That is, the position of the Earth in its orbit is not an exact real number, such as $90$ degrees, but is instead an interval, such as $\cuhinterval{68.4}{111.6}$ representing a range of angles between $68.4$ degrees and $111.6$ degrees.\footnote{We use decimal numbers to represent exact rational numbers.  For example, $68.4$ is to be understood as an abbreviation for $\frac{684}{10}$.}
Similarly, time is measured in discrete intervals of length $0.12$ years.  The earliest time interval in the model is near the year $-2000$ and the latest time interval is near the year $2000$.  Moreover, this model is faithful---it is in exact agreement with all observations.

There are ten possible measurements for the angular position of the Earth in Model~\ref{m:discrete}:
\begin{align*}
&\cuhinterval{32.4}{75.6} & &\cuhinterval{68.4}{111.6} & &\cuhinterval{104.4}{147.6} & &\cuhinterval{140.4}{183.6} & &\cuhinterval{176.4}{219.6}\\
&\cuhinterval{212.4}{255.6} & &\cuhinterval{248.4}{291.6} & &\cuhinterval{284.4}{327.6} & &\cuhinterval{320.4}{3.6} & &\cuhinterval{356.4}{39.6}
\end{align*}
These are the intervals obtained by dividing the $360$ degrees of the circle into ten equal intervals of $36$ degrees each, then extending each interval by exactly $3.6$ degrees on both sides, bringing the total length of each interval to $43.2$ degrees.  Therefore, consecutive intervals overlap by $7.2$ degrees (there is also overlap in consecutive time intervals), and this serves an important purpose.  The Earth's orbit is not, in reality, a perfect circle, and the Earth does not spend an equal amount of time in each of the intervals.  But because the eccentricity of the Earth's orbit contributes to, at most, only about a $2$ degree deviation~\cite{jE98} from the simple model of planetary motion (Model~\ref{m:real}), the overlap of these intervals is more than adequate to conceal evidence of the eccentricity, ensuring that this discrete model is faithful.  Also note that the overlap is a realistic feature of all known instruments which measure angles, since each such instrument has only a limited accuracy.  If angles are measured with a protractor, for example, the accuracy might be limited by the thickness of the lines painted on the protractor, which divide one reading from another.  For example, if the lines are $7.2$ degrees thick, then it might not be possible to distinguish a reading of $\cuhinterval{32.4}{75.6}$ from a reading of $\cuhinterval{68.4}{111.6}$ if the quantity being measured is somewhere on that line (that is, if the quantity is somewhere between $68.4$ and $75.6$ degrees).  The accuracy of measuring instruments is discussed in greater detail in Section~\ref{s:accuracy}.

\section{Non-Discrete Continuous Planetary Motion} \label{s:continuous}

Many commonly-studied computable physical models are discrete, but non-discrete continuous models are also possible.  For example, a non-discrete continuous computable physical model of planetary motion is the following.
\begin{cuhmodel}[Non-Discrete Continuous Planetary Motion] \label{m:continuous}
Let $S$ be the set of all pairs $\bigl\langle\cuhinterval{r}{s},\linebreak[0]\cuhinterval{p}{q}\bigr\rangle$ such that
\begin{align*}
r&=\frac{i}{10^n}-\frac{1}{10^{n+1}} & p&=360\bigl(r-\lfloor r\rfloor\bigr)\\
s&=\frac{i+1}{10^n}+\frac{1}{10^{n+1}} & q&=360\bigl(s-\lfloor s\rfloor\bigr)
\end{align*}
for some integer $i$ and some positive integer $n$.  The angular position of the Earth, represented as a range of angles measured in degrees, is given by the function $\alpha\bigl\langle\cuhinterval{r}{s},\cuhinterval{p}{q}\bigr\rangle=\cuhinterval{p}{q}$.  The time interval, measured in years, is given by the function $\tau\bigl\langle\cuhinterval{r}{s},\cuhinterval{p}{q}\bigr\rangle=\cuhinterval{r}{s}$.
\end{cuhmodel}

Like the discrete model, angular position and time are measured in intervals, but in this case the intervals are not all the same length.  In particular, there are arbitrarily small intervals for the observable quantities of position and time, meaning that these quantities may be measured to arbitrary precision.  This feature of Model~\ref{m:continuous} allows us to speak about real-valued positions and times, despite the fact that the values of observable quantities in the model are all non-negative integers, not real numbers.

This is because a real number is not the result of a single measurement, but is instead the limit of a potentially-infinite sequence of measurements.  Suppose, for example, that we wish to measure the circumference of a circle whose diameter is exactly one meter.  Measured with unmarked metersticks, we measure the circumference to be $3$ meters.  If the sticks are marked with millimeters, then we measure the circumference to be about $3.141$ meters.  And if they are marked with micrometers, then we measure a circumference of about $3.141592$ meters.  If we continue this process indefinitely with increasingly precise measuring instruments, then in the infinite limit we approach the real number $\pi$.

More formally, for each real number $x$ there is an infinite sequence of nested intervals $\cuhinterval{a_0}{b_0}$, $\cuhinterval{a_1}{b_1}$, $\cuhinterval{a_2}{b_2}$,~$\ldots$ that converges to $x$.  Given such a sequence, the function $\phi$ such that $\phi(n)=\cuhinterval{a_n}{b_n}$ for each non-negative integer $n$ is said to be an \emph{oracle}\index{oracles for real numbers} for $x$.  Note that there is more than one distinct sequence of nested intervals converging to $x$, and therefore more than one oracle for each $x$.  Of particular importance is the \emph{standard\index{oracles for real numbers!standard decimal} decimal oracle} $\cuhomicron_x$ for the real number $x$.  By definition, $\cuhomicron_x(n)=\cuhinterval{a_n}{b_n}$, where
\begin{align*}
a_n&=\frac{\bigl\lfloor 10^{n+1}x\bigr\rfloor}{10^{n+1}}-\frac{c}{10^{n+1}} & b_n&=\frac{\bigl\lfloor 10^{n+1}x\bigr\rfloor+1}{10^{n+1}}+\frac{c}{10^{n+1}}
\end{align*}
for each non-negative integer $n$, and where the \emph{accuracy\index{accuracy factors} factor} $c$ is a positive rational number constant.  We say that $x$ is a \emph{recursive real number} if and only if $\cuhomicron_x$ is a recursive function.  Note that not all real numbers are recursive~\cite{aT36}.

Now, returning to Model~\ref{m:continuous}, suppose that we are asked to find the position of the Earth at some real-valued time $t$.  Suppose further that we are given the oracle $\cuhomicron_t$ with accuracy factor $c=\frac{1}{10}$.  Note that as we increase $n$, the values $\cuhomicron_t(n)$ are increasingly precise measurements of the time $t$ in Model~\ref{m:continuous}.  Therefore, for each $n$ there is some state $\bigl\langle \cuhomicron_t(n),\cuhinterval{p_n}{q_n}\bigr\rangle$ in the set $S$ of Model~\ref{m:continuous}.  Because $S$ is a recursive set, and because there is exactly one state corresponding to each time measurement, the function $\epsilon$ such that $\epsilon(n)=\cuhinterval{p_n}{q_n}$ is a recursive function relative to the oracle $\cuhomicron_t$.  In fact, if $\cuhomicron_t(n)=\cuhinterval{r_n}{s_n}$, then
\begin{equation*}
\epsilon(n)=\cuhBinterval{360\bigl(r_n-\lfloor r_n\rfloor\bigr)}{360\bigl(s_n-\lfloor s_n\rfloor\bigr)}
\end{equation*}
for each non-negative integer $n$.  And since the sequence of intervals $\cuhinterval{r_0}{s_0}$, $\cuhinterval{r_1}{s_1}$, $\cuhinterval{r_2}{s_2}$,~$\ldots$ converges to $t$, it immediately follows that the sequence of intervals $\epsilon(0)$, $\epsilon(1)$, $\epsilon(2)$,~$\ldots$ converges to $a=360\bigl(t-\lfloor t\rfloor\bigr)$ whenever $t$ is not an integer.  In other words, $\epsilon$ is an oracle for the angular position $a$.

But in the case that $t$ is an integer,
\begin{equation*}
\epsilon(n)=\cuhBinterval{360-\frac{36}{10^{n+1}}}{\frac{396}{10^{n+1}}}
\end{equation*}
for all non-negative integers $n$, and $\cuhinterval{356.4}{39.6}$, $\cuhinterval{359.64}{3.96}$, $\cuhinterval{359.964}{0.396}$,~$\ldots$ is the resulting sequence.  In the standard topology of the real numbers an interval $\cuhinterval{x}{y}$ should have $x<y$, so the question of whether or not this sequence converges to a point $a$ in that standard topology cannot be meaningfully answered.  But if we are willing to abandon the standard topology of the real numbers, then we may conventionally define this sequence to converge to $a=0$. In fact, this definition is tantamount to establishing the topology of a circle of circumference 360 for all angles $a$.\footnote{\label{n:angle}A basis for this topology is represented by the set of all possible angle measurements.  In particular, if $x<y$ then $\cuhinterval{x}{y}$ represents the set of all real numbers $a$ such that $x<a<y$, and if $x>y$ then $\cuhinterval{x}{y}$ represents the set of all real numbers $a$ such that $0\leq a<y$ or $x<a<360$.}
Of course, this definition is justified since the readings after $360$ on a measuring instrument for angles are identified with those readings after $0$.  In other words, angles really do lie in a circle.

So, given the oracle $\cuhomicron_t$ for a real-valued time $t$, Model~\ref{m:continuous} allows us to compute an oracle $\epsilon$ for the angular position $a$ of the Earth at that time.  These predictions are in complete agreement with the predictions of the simple model of planetary motion (Model~\ref{m:real}).  In fact, imposing the appropriate topology on the space of angles $a$, the mapping from $t$ to $a$ in Model~\ref{m:continuous} is continuous.  The same mapping is discontinuous in the standard topology of the real numbers, which leads Kreisel's criterion to fail for Model~\ref{m:real}.  The formulation of computable physical models on effective topological spaces is discussed in greater detail in Sections~\ref{s:oracles} through \ref{s:kreisel}.

\section{Coarse-Graining} \label{s:coarse}

Observable quantities in computable physical models are defined operationally.  This means that each observable quantity is defined so as to correspond to a specific physical operation, such as the operation of comparing a length to the markings on a meterstick (where the meterstick itself is constructed according to a prescribed operation).  This is problematic for the non-discrete continuous model of planetary motion (Model~\ref{m:continuous}) because, for example, arbitrary precision angle measurements are made with a single observable quantity in the model.  That is, to assert that a model such as Model~\ref{m:continuous} is faithful, one must assert that there exists an operation which is capable of measuring angles to arbitrary precision.  It is not known whether or not such an operation actually exists.  And although the point is somewhat moot, since Model~\ref{m:continuous} is clearly not faithful, it raises the question of whether this is an accidental feature of Model~\ref{m:continuous}, or whether it is a feature common to all non-discrete continuous computable physical models.

A more practical alternative to Model~\ref{m:continuous} might introduce an infinite sequence of observable quantities $\alpha_1$, $\alpha_2$, $\alpha_3$,~$\ldots$~, each with finitely many digits of precision, and each more precise than its predecessor in the sequence.  In this case, given a state $s$, the values $\alpha_1(s)$, $\alpha_2(s)$, $\alpha_3(s)$,~$\ldots$ would form a sequence of intervals converging to a real number representing the angular position of the Earth in that state.  But a computable physical model has only countably many states, and there are uncountably many real numbers in the interval $\cuhinterval{0}{360}$.  Therefore, there must be some real number position in the interval $\cuhinterval{0}{360}$ that the Earth never attains.\footnote{In particular, this is a real number constructed by diagonalizing over those real numbers which are associated with each of the countably many states.}
That is, a computable physical model of this alternative form is not continuous in the intended topology.

Rather than considering arbitrary precision measurements, let us introduce just one additional level of precision into the discrete model of planetary motion (Model
\ref{m:discrete}).
\begin{cuhmodel} \label{m:high}
Let $S$ be the set of all quadruples $\bigl\langle\cuhinterval{r_1}{s_1},\cuhinterval{r_2}{s_2},\cuhinterval{p_1}{q_1},\cuhinterval{p_2}{q_2}\bigr\rangle$ such that
\begin{align*}
r_1&=\frac{i}{10}-\frac{1}{10^2} & p_1&=360\bigl(r_1-\lfloor r_1\rfloor\bigr)\\
s_1&=\frac{i+1}{10}+\frac{1}{10^2} & q_1&=360\bigl(s_1-\lfloor s_1\rfloor\bigr)\\
r_2&=\frac{j}{10^2}-\frac{1}{10^3} & p_2&=360\bigl(r_2-\lfloor r_2\rfloor\bigr)\\
s_2&=\frac{j+1}{10^2}+\frac{1}{10^3} & q_2&=360\bigl(s_2-\lfloor s_2\rfloor\bigr)
\end{align*}
for some integers $i$ and $j$ with $10i\leq j\leq 10i+9$.  The angular position of the Earth, represented as a range of angles measured in degrees with a low-precision measuring instrument, is given by the function
\begin{equation*}
\alpha_1\bigl\langle\cuhinterval{r_1}{s_1},\cuhinterval{r_2}{s_2},\cuhinterval{p_1}{q_1},\cuhinterval{p_2}{q_2}\bigr\rangle=\cuhinterval{p_1}{q_1}
\end{equation*}
The angular position of the Earth, represented as a range of angles measured in degrees with a high-precision measuring instrument, is given by the function
\begin{equation*}
\alpha_2\bigl\langle\cuhinterval{r_1}{s_1},\cuhinterval{r_2}{s_2},\cuhinterval{p_1}{q_1},\cuhinterval{p_2}{q_2}\bigr\rangle=\cuhinterval{p_2}{q_2}
\end{equation*}
The time interval, measured in years by a low-precision measuring instrument, is given by the function
\begin{equation*} \tau_1\bigl\langle\cuhinterval{r_1}{s_1},\cuhinterval{r_2}{s_2},\cuhinterval{p_1}{q_1},\cuhinterval{p_2}{q_2}\bigr\rangle=\cuhinterval{r_1}{s_1}
\end{equation*}
The time interval, measured in years by a high-precision measuring instrument, is given by the function
\begin{equation*} \tau_2\bigl\langle\cuhinterval{r_1}{s_1},\cuhinterval{r_2}{s_2},\cuhinterval{p_1}{q_1},\cuhinterval{p_2}{q_2}\bigr\rangle=\cuhinterval{r_2}{s_2}
\end{equation*}
\end{cuhmodel}

Note that if the high-precision observable quantities $\alpha_2$ and $\tau_2$ are ignored, then the predictions of Model~\ref{m:high} agree exactly with the predictions of Model~\ref{m:discrete}.\footnote{But it should be noted that the model obtained by omitting $\alpha_2$ and $\tau_2$ from Model~\ref{m:high} is not identical to Model~\ref{m:discrete}.  In particular, for each state in Model~\ref{m:discrete}, there are ten indistinguishable states in the model obtained by omitting $\alpha_2$ and $\tau_2$ from Model~\ref{m:high}.  That is, these models are not isomorphic.  See Section~\ref{s:theorems}.}
The process of removing observable quantities from a model to obtain a new model with fewer observable quantities is called \emph{coarse-graining}\index{coarse-graining}.  But while Model~\ref{m:discrete} is faithful, Model~\ref{m:high} is not faithful---physical measurements do not agree with the values of the observable quantities $\alpha_2$ and $\tau_2$ because the orbit of the Earth is not a perfect circle.

A traditional conception of science regards all physical models as inexact approximations of reality, and holds that the goal of science is to produce progressively more accurate models whose predictions more closely match observations than the predictions of previous models.  That conception of science is reasonable when the values of observable quantities are real numbers, since the real numbers predicted by physical models are never exactly the same as the real numbers `measured' in the laboratory.  But when non-negative integers are used for the values of observable quantities, then an alternate conception of science is possible.

In this alternate conception there exist faithful models that are in exact agreement with reality, but perhaps only for a small subset of all physically observable quantities.  For example, Model~\ref{m:discrete} is faithful, but only predicts the angular position of the Earth to within $43.2$ degrees, and only for a limited range of times.  The goal of science is then to produce \emph{more refined} models.  That is, the goal of science is to discover faithful models which have larger sets of observable quantities, and are therefore capable of predicting increasing numbers of facts.

\section{Radioactive Decay} \label{s:radio}

Given non-negative integers $x$ and $y$, let $\beta(x,y)$ be the length $y$ sequence composed of the first $y$ bits in the binary expansion of $x$.  For example $\beta(13,6)=\langle 0,0,1,1,0,1\rangle$.  Now suppose that a single atom of a radioactive isotope, such as nitrogen-13, is placed inside a detector at time $t=0$.  We say that the detector has status $1$ if it has detected the decay of the isotope, and has status $0$ otherwise.  The \emph{history} of the detector at time $t$ is the length $t$ sequence of bits corresponding to the status of the detector at times $1$ through $t$.  For example, if the isotope decays sometime between $t=2$ and $t=3$, then the history of the detector at time $t=5$ is $\langle 0,0,1,1,1\rangle$.  The following computable physical model models the status of the detector as a function of time.
\begin{cuhmodel}[Radioactive Decay]
Let $S$ be the set of all triples $\bigl\langle t,\beta(2^n-1,t),j\bigr\rangle$ where $n$, $t$, and $j$ are non-negative integers such that $n\leq t$, $t\ne0$, and $2j\leq 2^n-1$.  The history of the detector is given by the function $\eta\langle t,h,j\rangle=h$, and the time, measured in units of the half-life of the isotope, is given by the function $\tau\langle t,h,j\rangle=t$.
\end{cuhmodel}

This is a model of the many-worlds\index{quantum mechanics!many-worlds interpretation of} interpretation~\cite{DG73} of radioactive decay.  Suppose that one asks, ``What will the status of the detector be at time $t=2$?''  There are four states $\langle t,h,j\rangle$ such that $\tau\langle t,h,j\rangle=2$, namely
\begin{align*}
\bigl\langle 2&,\langle 0,0\rangle,0\bigr\rangle & \bigl\langle 2&,\langle 0,1\rangle,0\bigr\rangle & \bigl\langle 2&,\langle 1,1\rangle,0\bigr\rangle & \bigl\langle 2&,\langle 1,1\rangle,1\bigr\rangle
\end{align*}
In three of these states, the detector has status $1$, and in one state it has status $0$.  If we assume that each state of the system is equally likely, then there is a $\frac{3}{4}$ probability that the detector will have status $1$ at time $t=2$.  But if we ask, ``If the detector has status $1$ at time $t=1$, then what will its status be at time $t=2$?'' The answer is ``$1$'', since the detector has status $1$ at time $2$ in both states where the detector had status $1$ at time $1$.  These results are in agreement with conventional theory.

\section{Ensembles of Physical Models}

Suppose that a planet orbits a distant star and that we are uncertain of the planet's orbital period.  In particular, suppose that we believe its motion is faithfully described by either the discrete model of planetary motion (Model~\ref{m:discrete}) or by the following computable physical model.
\begin{cuhmodel} \label{m:slow}
Let $S$ be the set of all pairs $\bigl\langle\cuhinterval{r}{s},\cuhinterval{p}{q}\bigr\rangle$ such that
\begin{align*}
r&=\frac{i}{10}-\frac{1}{100} & p&=360\biggl(\frac{n}{10}-\frac{1}{100}-\Bigl\lfloor\frac{n}{10}-\frac{1}{100}\Bigr\rfloor\biggr)\\
s&=\frac{i+1}{10}+\frac{1}{100} & q&=360\biggl(\frac{n+1}{10}+\frac{1}{100}-\Bigl\lfloor\frac{n+1}{10}+\frac{1}{100}\Bigr\rfloor\biggr)
\end{align*}
for some integer $i$ between $-20000$ and $20000$, and such that $n=\lfloor i/4\rfloor$.  The angular position of the planet, represented as a range of angles measured in degrees, is given by the function $\alpha\bigl\langle\cuhinterval{r}{s},\cuhinterval{p}{q}\bigr\rangle=\cuhinterval{p}{q}$.  The time interval, measured in Earth years, is given by the function $\tau\bigl\langle\cuhinterval{r}{s},\cuhinterval{p}{q}\bigr\rangle=\cuhinterval{r}{s}$.
\end{cuhmodel}
\noindent
Note that this model is similar to Model~\ref{m:discrete}, except that the orbital period of the planet is $4$ Earth years, rather than $1$ Earth year.

If for each of the two models we are given a rational number expressing the probability that that model is faithful, then a \emph{statistical\index{ensembles!statistical} ensemble} of the models may be constructed.  For example, if Model~\ref{m:discrete} is twice as likely as Model~\ref{m:slow}, then a corresponding statistical ensemble is the following.  Note that this statistical ensemble is itself a computable physical model.
\begin{cuhmodel}[Ensemble of Models] \label{m:ensemble}
Let $S$ be the set of all triples $\bigl\langle\cuhinterval{r}{s},\cuhinterval{p}{q},j\bigr\rangle$ such that
\begin{align*}
r&=\frac{i}{10}-\frac{1}{100} & p&=360\biggl(\frac{n}{10}-\frac{1}{100}-\Bigl\lfloor\frac{n}{10}-\frac{1}{100}\Bigr\rfloor\biggr)\\
s&=\frac{i+1}{10}+\frac{1}{100} & q&=360\biggl(\frac{n+1}{10}+\frac{1}{100}-\Bigl\lfloor\frac{n+1}{10}+\frac{1}{100}\Bigr\rfloor\biggr)
\end{align*}
for some integer $i$ between $-20000$ and $20000$, where $j=0$, $1$, or $2$, and where
\begin{equation*}
n=
\begin{cases}
i &\text{if $j=0$ or $1$}\\
\lfloor i/4\rfloor &\text{if $j=2$}
\end{cases}
\end{equation*}
The angular position of the planet, represented as a range of angles measured in degrees, is given by the function $\alpha\bigl\langle\cuhinterval{r}{s},\cuhinterval{p}{q},j\bigr\rangle=\cuhinterval{p}{q}$.  The time interval, measured in Earth years, is given by the function $\tau\bigl\langle\cuhinterval{r}{s},\cuhinterval{p}{q},j\bigr\rangle=\cuhinterval{r}{s}$.
\end{cuhmodel}
\noindent
Since Model~\ref{m:discrete} is twice as likely as Model~\ref{m:slow}, there are two states, $\bigl\langle\cuhinterval{r}{s},\cuhinterval{p}{q},0\bigr\rangle$ and $\bigl\langle\cuhinterval{r}{s},\cuhinterval{p}{q},1\bigr\rangle$ in the ensemble for each state $\bigl\langle\cuhinterval{r}{s},\cuhinterval{p}{q}\bigr\rangle$ in Model~\ref{m:discrete}, and there is one state $\bigl\langle\cuhinterval{r}{s},\cuhinterval{p}{q},2\bigr\rangle$ in the ensemble for each state $\bigl\langle\cuhinterval{r}{s},\cuhinterval{p}{q}\bigr\rangle$ in Model~\ref{m:slow}.  Note that the index $j$ in each state $\bigl\langle\cuhinterval{r}{s},\cuhinterval{p}{q},j\bigr\rangle$ of the ensemble is not observable.

Now, if we ask for the position of the planet during the time interval $\cuhinterval{0.29}{0.41}$, for example, there are three possible states $\bigl\langle\cuhinterval{r}{s},\cuhinterval{p}{q},j\bigr\rangle$ in the ensemble such that
\begin{equation*}
\tau\bigl\langle\cuhinterval{r}{s},\cuhinterval{p}{q},j\bigr\rangle=\cuhinterval{0.29}{0.41}
\end{equation*}
namely
\begin{align*}
&\bigl\langle\cuhinterval{0.29}{0.41},\cuhinterval{104.4}{147.6},0\bigr\rangle\\
&\bigl\langle\cuhinterval{0.29}{0.41},\cuhinterval{104.4}{147.6},1\bigr\rangle\\ &\bigl\langle\cuhinterval{0.29}{0.41},\cuhinterval{356.4}{39.6},2\bigr\rangle
\end{align*}
Since the planet's angular position is $\cuhinterval{104.4}{147.6}$ for two of these three states, the position measurement $\cuhinterval{104.4}{147.6}$ has a probability of $\frac{2}{3}$.  Similarly, because the planet's angular position is $\cuhinterval{356.4}{39.6}$ for one of the three states, the position measurement $\cuhinterval{356.4}{39.6}$ has a probability of $\frac{1}{3}$.  These probabilities are a direct reflection of our uncertainty about which of the two underlying physical models, Model~\ref{m:discrete} or Model~\ref{m:slow}, is the true faithful model.  In particular, because Model~\ref{m:discrete} has been deemed twice as likely as Model~\ref{m:slow}, the position of the planet in Model~\ref{m:discrete}, namely $\cuhinterval{104.4}{147.6}$, has twice the probability of the position predicted by Model~\ref{m:slow}, namely $\cuhinterval{356.4}{39.6}$.

It is important to note that there is no observable quantity corresponding to probability in Model~\ref{m:ensemble}.  Instead, probability is a mathematical tool used to interpret the model's predictions.  This sort of interpretation of an ensemble of models is appropriate whenever the ensemble is composed from all possible models which could describe a particular system, with the number of copies of states of the individual models reflecting our confidence in the predictions of those models.  See reference~\cite{eJ57} for a more detailed account of this subjectivist\index{probability!subjectivist interpretation of} interpretation of probability in physics.

Ensembles may be constructed in other circumstances as well, and we may refer to such ensembles as \emph{non-statistical\index{ensembles!non-statistical} ensembles} of physical models.  Non-statistical ensembles of physical models are commonplace in the sciences.  For example, they result whenever a constant, such as an initial position, is left unspecified in the statement of a model.  That model can then be used to describe any member of a family of systems, each of which may have a different value for the constant.  But most importantly, when a non-statistical ensemble of physical models is constructed, no claims as to the likelihood of one value of the constant, as compared to some other value of the constant, are being made.  In fact, this is the defining characteristic of a non-statistical ensemble of models.  Non-statistical ensembles can be useful because they provide a convenient way to collect together sets of closely-related models.

\section{Incompatible Measurements}

A pair of measurements is said to be \emph{simultaneous} if and only if they are both performed while the system is in a single state.  An essential feature of quantum mechanical systems is that there may be quantities which are not simultaneously measurable.  For example, the measurement of one quantity, such as the position of a particle, might affect the subsequent measurement of another quantity, such as the particle's momentum.  Such measurements are said to be \emph{incompatible}.  It is natural to ask whether computable physical models can be used to describe systems which feature incompatible measurements.

Discrete quantum mechanical systems are often formalized as follows~\cite{pD30,jN32}.  The \emph{quantum mechanical state} of a system is a normalized vector $v$ in some normed complex vector space $V$.  Typically, $V$ is a Hilbert\index{Hilbert, David} space and $v$ is a \emph{wave function}.  For each \emph{quantum mechanical measurement} there is a corresponding set $B=\{v_1,v_2,v_3,\ldots\}$ of normalized basis vectors for $V$.  Each member of $B$ corresponds to a possible value of the measurement.  Because $B$ is a basis for $V$, $v=a_1v_1+a_2v_2+a_3v_3+\cdots$ for some complex numbers $a_1$, $a_2$, $a_3$,~$\ldots$.  If the system is in quantum mechanical state $v$ and no two members of $B$ correspond to the same measurement value,\footnote{Alternatively, if $v_{n_1}$, $v_{n_2}$, $v_{n_3}$,~$\ldots$ are distinct basis vectors corresponding to the same measurement value, then the probability of measuring the value is $\lvert a_{n_1}\rvert^2+\lvert a_{n_2}\rvert^2+\lvert a_{n_3}\rvert^2+\cdots$.  If that value is actually measured, then the state of the system immediately after the measurement is the normalization of $a_{n_1}v_{n_1}+a_{n_2}v_{n_2}+a_{n_3}v_{n_3}+\cdots$.  See reference~\cite{aM59}.}
then the probability that the measurement will have the value corresponding to $v_n$ is $\lvert a_n\rvert^2$.  In this case, if the actual value which is measured is the value corresponding to $v_n$, then the quantum mechanical state of the system immediately after that measurement is $v_n$.  The state $v$ is said to have \emph{collapsed} to $v_n$.  During the time between measurements, the quantum mechanical state of a system may evolve according to a rule such as Schr\"odinger's\index{Schr\"odinger, Erwin} equation.

Consider, for example, the problem of measuring the components of the spin of an isolated electron.  In this case, $V$ is the set of all vectors $(a,b)$ such that $a$ and $b$ are complex numbers, where the norm $\lVert(a,b)\rVert$ is defined to be $\sqrt{\lvert a\rvert^2+\lvert b\rvert^2}$.  A quantum mechanical measurement of the $z$ component of the electron's spin has two possible values, $-\frac{1}{2}\hbar$ and $+\frac{1}{2}\hbar$.  The basis vectors corresponding to these values are $(0,1)$ and $(1,0)$, respectively.  The quantum mechanical measurement of another component of the electron's spin, lying in the $xz$ plane at an angle of $60$ degrees to the $z$ axis, also has two possible values, $-\frac{1}{2}\hbar$ and $+\frac{1}{2}\hbar$.  The basis vectors corresponding to these values are $\bigl(-\frac{1}{2},\frac{\sqrt{3}}{2}\bigr)$ and $\bigl(\frac{\sqrt{3}}{2},\frac{1}{2}\bigr)$, respectively.  So, for example, if the spin component in the $z$ direction is measured to have a value of $+\frac{1}{2}\hbar$ at time $t=0$, then since
\begin{equation*} (1,0)=-\frac{1}{2}\Bigl(-\frac{1}{2},\frac{\sqrt{3}}{2}\Bigr)+\frac{\sqrt{3}}{2}\Bigl(\frac{\sqrt{3}}{2},\frac{1}{2}\Bigr)
\end{equation*}
there is a $\big\lvert\frac{\sqrt{3}}{2}\big\rvert^2=\frac{3}{4}$ probability that if the $60$-degree electron spin component is measured at time $t=1$, then that component will also have a value of $+\frac{1}{2}\hbar$.

Supposing that the $60$-degree electron spin component is measured to have a value of $+\frac{1}{2}\hbar$ at time $t=1$, a similar line of reasoning implies that if the spin's $z$ component is measured at time $t=2$, then there is a $\big\lvert\frac{1}{2}\big\rvert^2=\frac{1}{4}$ probability that the value of that measurement will be $-\frac{1}{2}\hbar$, since
\begin{equation*} \Bigl(\frac{\sqrt{3}}{2},\frac{1}{2}\Bigr)=\frac{1}{2}(0,1)+\frac{\sqrt{3}}{2}(1,0)
\end{equation*}
Therefore, if the $z$ component of the electron's spin is measured at time $t=0$, followed by a measurement of the $60$-degree spin component at time $t=1$, and followed by another measurement of the $z$ component at time $t=2$, then the values of the two measurements of the $z$ component need not be the same.  Indeed, the quantum mechanical state of the system does not change\footnote{In this case, the quantum mechanical state of the system does not change between measurements because the electron is isolated.  For example, the electron is free from external electromagnetic fields or other influences that might cause its spin to precess.}
between times $t=0$ and $t=1$, or between times $t=1$ and $t=2$, but the measurement of the $60$-degree spin component at time $t=1$ disturbs the system and can potentially change the value of any subsequent measurement of the $z$ component.  That is, measurement of the electron's $60$-degree spin component is incompatible with measurement of its $z$ component.

Let us formalize this system as a computable physical model.  The system is composed of the electron, the apparatus used to make the quantum mechanical measurements, and the researcher who chooses which components to measure.\footnote{We refrain from asking questions about the probability with which the researcher chooses which components to measure.  That is, this model describes a non-statistical ensemble of researchers.}
We assume that the quantum mechanical state of the electron is $(1,0)$ at time $t=0$, and that the researcher makes subsequent quantum mechanical measurements of the electron's spin components at times $t=1$ and $t=2$.  When a quantum mechanical measurement is performed, a record is made (perhaps in the researcher's notebook) of the value of this measurement and of the component that was measured.  We construct the computable physical model of this system from the point of view of an agent who observes only this recorded history and the time.
\begin{cuhmodel}[Electron Spin Measurement] \label{m:spin}
Let $S$ be the set of all triples $\langle t,h,j\rangle$ such that
\begin{align*}
&\begin{aligned}
t&=1\\
h&=\cuhinterval{0}{+1}\\
j&=0
\end{aligned} & &\text{or} &
&\begin{aligned}
t&=1\\
h&=\cuhinterval{60}{-1}\\
j&=1
\end{aligned} & &\text{or} &
&\begin{aligned}
t&=1\\
h&=\cuhinterval{60}{+1}\\
j&=2+m
\end{aligned}
\end{align*}
or
\begin{align*}
&\begin{aligned}
t&=2\\
h&=\bigl\langle\cuhinterval{0}{+1},\cuhinterval{0}{+1}\bigr\rangle\\
j&=5
\end{aligned} & &\text{or} &
&\begin{aligned}
t&=2\\
h&=\bigl\langle\cuhinterval{0}{+1},\cuhinterval{60}{-1}\bigr\rangle\\
j&=6
\end{aligned}\\
\intertext{or}\\
&\begin{aligned}
t&=2\\
h&=\bigl\langle\cuhinterval{0}{+1},\cuhinterval{60}{+1}\bigr\rangle\\
j&=7+m
\end{aligned} & &\text{or} &
&\begin{aligned}
t&=2\\
h&=\bigl\langle\cuhinterval{60}{-1},\cuhinterval{0}{-1}\bigr\rangle\\
j&=10
\end{aligned}\\
\intertext{or}\\
&\begin{aligned}
t&=2\\
h&=\bigl\langle\cuhinterval{60}{-1},\cuhinterval{0}{+1}\bigr\rangle\\
j&=11+m
\end{aligned} & &\text{or} &
&\begin{aligned}
t&=2\\
h&=\bigl\langle\cuhinterval{60}{-1},\cuhinterval{60}{-1}\bigr\rangle\\
j&=14
\end{aligned}\\
\intertext{or}\\
&\begin{aligned}
t&=2\\
h&=\bigl\langle\cuhinterval{60}{+1},\cuhinterval{0}{-1}\bigr\rangle\\
j&=15+n
\end{aligned} & &\text{or} &
&\begin{aligned}
t&=2\\
h&=\bigl\langle\cuhinterval{60}{+1},\cuhinterval{0}{+1}\bigr\rangle\\
j&=24+m
\end{aligned}
\end{align*}
or
\begin{align*}
t&=2\\
h&=\bigl\langle\cuhinterval{60}{+1},\cuhinterval{60}{+1}\bigr\rangle\\
j&=27+m
\end{align*}
for some integers $m$ and $n$ with $0\leq m\leq 2$ and $0\leq n\leq 8$.  The time is given by the function $\tau\langle t,h,j\rangle=t$.  The history is given by the function $\eta\langle t,h,j\rangle=h$.  A history is a chronological sequence of records, with the leftmost record being the oldest.  Each record is a pair $\cuhinterval{a}{b}$ of rational numbers, where $a$ is the angle from the $z$ axis, measured in degrees, of a component of the electron's spin, and where $b$ is the value of that component, measured in units of $\frac{1}{2}\hbar$.
\end{cuhmodel}
\noindent
Note that each state $\langle t,h,j\rangle$ has a distinct index $j$, which we will use to identify that particular state.

Model~\ref{m:spin} corresponds to the quantum mechanical system in the following sense.  First, the quantum mechanical state of the system at time $t$ corresponds to a set of states in the computable physical model.  For example, if the researcher decides to measure the $60$-degree component of the electron's spin at time $t=1$, then the quantum mechanical state of the system is represented by the set of states with indices $1$ through $4$.  Assuming that the states in the set are equally likely, there is a $\frac{3}{4}$ probability that this component will have a measured value of $+\frac{1}{2}\hbar$, for example.  Immediately after the measurement is made, the quantum mechanical state collapses, becoming either the set of states with indices $2$ through $4$, or the singleton set containing only the state with index $1$.  The collapse occurs because the information provided by the quantum mechanical measurement allows us to identify the state of the system more precisely, eliminating those states which disagree with the measurement result.\footnote{For a more detailed discussion of this ensemble\index{quantum mechanics!ensemble interpretation of} interpretation of the collapse of a quantum mechanical state, see reference~\cite{dB51}.}
The quantum mechanical state then evolves to a new set of states at time $t=2$.  For example, if the measured value of the $60$-degree electron spin component is $+\frac{1}{2}\hbar$ at time $t=1$, and if the researcher plans to measure the $0$-degree electron spin component (that is, the $z$ component) at time $t=2$, then the quantum mechanical state immediately before that measurement at time $t=2$ is the set of states with indices $15$ through $26$.

Computable physical models similar to Model~\ref{m:spin} can be constructed for quantum mechanical systems which satisfy the following criteria.
\begin{enumerate}
\item There is a set of possible measurements $\{m_0,m_1,m_2,\ldots,m_i,\ldots\}$ indexed by non-negative integers $i$.
\item Every discrete time step, one measurement from this set is performed.
\item The possible values of each measurement $m_i$ are identified with non-negative integers.
\item If $\phi(i,n,t,h)$ is the probability that the measurement with index $i$ has the value $n$, given that the measurement is performed at time step $t$ and that $h=\bigl\langle\cuhinterval{i_1}{n_1},\ldots,\cuhinterval{i_{t-1}}{n_{t-1}}\bigr\rangle$  is the history of past measurements and their values, then $\phi(i,n,t,h)$ is a rational number.
\item If there is no measurement with index $i$ or if the non-negative integer $n$ does not correspond to a value of the measurement with index $i$, then $\phi(i,n,t,h)=0$.
\item For each choice of non-negative integers $i$, $t$, and $h$, there are only finitely many non-negative integers $n$ such that $\phi(i,n,t,h)>0$.
\item $\phi$ is a recursive function.
\end{enumerate}
If a quantum mechanical system satisfies these criteria, then we can determine whether or not
\begin{equation*}
s=\bigl\langle t,\bigl\langle\cuhinterval{i_1}{n_1},\cuhinterval{i_2}{n_2},\ldots,\cuhinterval{i_t}{n_t}\bigr\rangle,j\bigr\rangle
\end{equation*}
is in the set $S$ of states of the corresponding computable physical model as follows.  First, if $t=0$, then $s$ is not in $S$.  Next, let $h_1=0$ and for each positive integer $k$ with $1<k\leq t$, let
\begin{equation*}
h_k=\bigl\langle\cuhinterval{i_1}{n_1},\cuhinterval{i_2}{n_2},\ldots,\cuhinterval{i_{k-1}}{n_{k-1}}\bigr\rangle
\end{equation*}
Now we perform the following calculations for each positive integer $k\leq t$.  If $\phi(i_k,n_k,k,h_k)=0$, then $s$ is not in $S$.  Otherwise, there must be finitely many non-negative integers $n$ such that the probability $\phi(i_k,n,k,h_k)$ is greater than zero.  Since probabilities must sum to $1$, those values for $n$ may be found exhaustively by calculating $\phi(i_k,0,k,h_k)$, $\phi(i_k,1,k,h_k)$, $\phi(i_k,2,k,h_k)$, and so on, until the the sum of these probabilities reaches $1$.  Let $d_k$ be the least common denominator of these rational probabilities, and let $a_k$ be the unique positive integer such that
\begin{equation*}
\phi(i_k,n_k,k,h_k)=\frac{a_k}{d_k}
\end{equation*}
If $j<a_1a_2\cdots a_t$, then $s$ is in $S$.  Otherwise, $s$ is not in $S$.

\section{The Accuracy of Measuring Instruments} \label{s:accuracy}

An important feature of the discrete model of planetary motion (Model~\ref{m:discrete}) is that the intervals representing time and angle measurements overlap.  The amount of overlap between adjacent intervals is determined by the accuracy of the corresponding measuring instrument.  The introduction of overlapping intervals is motivated by an argument such as the following.

If Model~\ref{m:discrete} were constructed using disjoint, non-overlapping intervals, then the states $\bigl\langle\cuhinterval{r}{s},\cuhinterval{p}{q}\bigr\rangle$ of that model would be given by
\begin{align*}
r&=i/10 & p&=360\bigl(r-\lfloor r\rfloor\bigr)\\
s&=(i+1)/10 & q&=360\bigl(s-\lfloor s\rfloor\bigr)
\end{align*}
where $i$ is an integer.  In particular, $\bigl\langle\cuhinterval{0.2}{0.3},\cuhinterval{72}{108}\bigr\rangle$ and $\bigl\langle\cuhinterval{0.3}{0.4},\cuhinterval{108}{144}\bigr\rangle$ would be two such states, with $\cuhinterval{r}{s}$ representing the state's time interval, measured in years, and with $\cuhinterval{p}{q}$ representing the corresponding interval of angular positions for the Earth, measured in degrees.  According to this model, if the position of the Earth is measured at time $t=0.298$ years, then $t$ is within the interval $\cuhinterval{0.2}{0.3}$, and the state of the system is $\bigl\langle\cuhinterval{0.2}{0.3},\cuhinterval{72}{108}\bigr\rangle$.  Therefore, according to this model, the position of the Earth should be between $72$ and $108$ degrees.  Indeed, the simple model of planetary motion (Model~\ref{m:real}) predicts that the angular position of the Earth at time $t=0.298$ years should be $360\bigl(0.298-\lfloor 0.298\rfloor\bigr)\approx 107$ degrees.  But the true position of the Earth in its orbit deviates from Model~\ref{m:real}.  In this case, the true position of the Earth at time $t=0.298$ years is about $109$ degrees,\footnote{This is assuming that time is measured in anomalistic years, with each year beginning at perihelion passage.  During the course of a year, the position of the Earth is the true anomaly, measured relative to that perihelion passage.}
which is outside the interval $\cuhinterval{72}{108}$.  Therefore, if the discrete model were constructed using disjoint, non-overlapping intervals, then the model would fail when $t=0.298$ years.

But the discrete model of planetary motion (Model~\ref{m:discrete}) was constructed using overlapping intervals.  In particular,
\begin{align*}
\bigl\langle\cuhinterval{0.19}{0.31}&,\cuhinterval{68.4}{111.6}\bigr\rangle & \bigl\langle\cuhinterval{0.29}{0.41}&,\cuhinterval{104.4}{147.6}\bigr\rangle
\end{align*}
are two states in Model~\ref{m:discrete}.  Note that at time $t=0.298$ years, Model~\ref{m:discrete} could be in either of these two states.  Furthermore, any pair of real-valued time $t$ and angle $a$ measurements which satisfy
\begin{equation*}
\big\lvert a-360\bigl(t-\lfloor t\rfloor\bigr)\big\rvert<7.2
\end{equation*}
fall within the time and angle intervals of some common state of Model~\ref{m:discrete}.  Since $\lvert a-360\bigl(t-\lfloor t\rfloor\bigr)\rvert$ is at most $2$ degrees~\cite{jE98} for all physically observed angles $a$ measured at times $t$, Model~\ref{m:discrete} is faithful.

It is important to point out, though, that Model~\ref{m:discrete} is faithful only if the results of measurements are uncertain when they occur within the region of overlap.  For example, at time $t=0.306$ years, two results of a time measurement are possible, $\cuhinterval{0.19}{0.31}$ and $\cuhinterval{0.29}{0.41}$, and an observer cannot be certain which of these intervals is the value of the measurement.  The actual angular position of the Earth at time $t=0.306$ years is about $112$ degrees, so $\cuhinterval{104.4}{147.6}$ is the only possible result of a position measurement.  Since
\begin{equation*}
\bigl\langle\cuhinterval{0.19}{0.31},\cuhinterval{104.4}{147.6}\bigr\rangle
\end{equation*}
is not one of the states of Model~\ref{m:discrete}, the observer is expected to realize, in retrospect, after measuring the angular position, that the true time measurement must have been $\cuhinterval{0.29}{0.41}$.  After providing a model for the phenomenon of accuracy, we will be able to reformulate Model~\ref{m:discrete} so that the results of measurements no longer possess this sort of ambiguity.

But first, note that the accuracy of a measuring instrument, by definition, can only be quantified relative to some other, more precise quantity.  For example, the argument above, concerning accuracy in Model~\ref{m:discrete}, makes frequent reference to exact real-valued angles and times.  Indeed, even when we express an angle measurement as an interval, such as $\cuhinterval{68.4}{111.6}$, we are implying that it is possible to distinguish an angle of $68.4$ degrees from an angle of $111.6$ degrees, and that other angles lie between those two values.  In principle, though, it is possible to describe the accuracy of a measuring instrument in a purely discrete manner, without any mention of real numbers.  For example, let us consider an instrument for measuring distances in meters, with the value of a measurement represented as an integer number of meters.  The accuracy of this measuring instrument can be quantified relative to a second instrument which measures distances in decimeters.

Presumably, the phenomenon of accuracy results from our inability to properly calibrate measuring instruments.  Although there are many different underlying causes of calibration error, it suffices to consider only one such cause for a simple model of this phenomenon.  We will suppose that when we measure a distance in meters, that we have difficulty aligning the measuring instrument with the origin, so that sometimes the instrument is aligned a decimeter too far in the negative direction, and at other times a decimeter too far in the positive direction.  Hence, there are two different physical models for the measurement.  In one model the instrument is misaligned in the negative direction, and in the other model it is misaligned in the positive direction.  Since we do not know which of these two models describes any one particular measurement, it is appropriate to combine them in the following statistical ensemble.
\begin{cuhmodel} \label{m:decimeters}
Let $S$ be the set of all triples $\bigl\langle\zeta(m),\zeta(d),\zeta(i)\bigr\rangle$ such that
\begin{equation*}
m=\Bigl\lfloor\frac{d+i}{10}\Bigr\rfloor
\end{equation*}
where $d$ is an integer, and where $i=-1$ or $+1$.  The distance, measured in meters, is given by the function
\begin{equation*}
\mu\bigl\langle\zeta(m),\zeta(d),\zeta(i)\bigr\rangle=\zeta(m)
\end{equation*}
The same distance, measured in decimeters, is given by the function
\begin{equation*}
\delta\bigl\langle\zeta(m),\zeta(d),\zeta(i)\bigr\rangle=\zeta(d)
\end{equation*}
\end{cuhmodel}
\noindent
Note that the function $\zeta$ was defined in Section~\ref{s:cuh}.  Also note that the index $i$ in each state $\bigl\langle\zeta(m),\zeta(d),\zeta(i)\bigr\rangle$ represents the calibration error, which is either $-1$ decimeter or $+1$ decimeter.  Model~\ref{m:decimeters} is a computable physical model.

A measurement of $d$ decimeters in Model~\ref{m:decimeters} can be interpreted as corresponding to an interval of $\cuhbinterval{\frac{d}{10}}{\frac{d+1}{10}}$ meters.  Note that a measurement of $9$ decimeters (corresponding to an interval of $\cuhinterval{0.9}{1.0}$ meters), for example, is possible in two distinct states of the model:
\begin{align*}
\bigl\langle\zeta(0),\zeta(9)&,\zeta(-1)\bigr\rangle & \bigl\langle\zeta(1),\zeta(9)&,\zeta(+1)\bigr\rangle
\end{align*}
Similarly, a measurement of $10$ decimeters (corresponding to an interval of $\cuhinterval{1.0}{1.1}$ meters) is possible in the states
\begin{align*}
\bigl\langle\zeta(0),\zeta(10)&,\zeta(-1)\bigr\rangle & \bigl\langle\zeta(1),\zeta(10)&,\zeta(+1)\bigr\rangle
\end{align*}
Hence, a measurement of $0$ meters overlaps with a measurement of $1$ meter on the intervals $\cuhinterval{0.9}{1.0}$ and $\cuhinterval{1.0}{1.1}$.  And in general, a measurement of $m$ meters overlaps with a measurement of $m+1$ meters on the intervals $\cuhinterval{m+0.9}{m+1.0}$ and $\cuhinterval{m+1.0}{m+1.1}$.  Therefore, a measurement of $m$ meters in Model~\ref{m:decimeters} can be understood as corresponding to an interval of $\cuhinterval{m-0.1}{m+1.1}$ meters, with adjacent intervals overlapping by $0.2$ meters.

Of course, this interpretation of Model~\ref{m:decimeters} presumes that decimeters can be measured with perfect accuracy.  A more realistic computable physical model can be constructed by supposing that decimeter measurements can also be misaligned, for example, by $-1$ centimeter or $+1$ centimeter.  Note that centimeters are treated as unobserved, purely theoretical constructions in this model---there is no observable quantity for centimeter measurements.
\begin{cuhmodel}
Let $S$ be the set of all quintuples $\bigl\langle\zeta(m),\zeta(d),\zeta(c),\zeta(i),\zeta(j)\bigr\rangle$ such that
\begin{align*}
m&=\Bigl\lfloor\frac{c+10i}{100}\Bigr\rfloor & d&=\Bigl\lfloor\frac{c+j}{10}\Bigr\rfloor
\end{align*}
where $c$ is an integer, $i=-1$ or $+1$, and $j=-1$ or $+1$.  The distance, measured in meters, is given by the function
\begin{equation*}
\mu\bigl\langle\zeta(m),\zeta(d),\zeta(c),\zeta(i),\zeta(j)\bigr\rangle=\zeta(m)
\end{equation*}
The same distance, measured in decimeters, is given by the function
\begin{equation*}
\delta\bigl\langle\zeta(m),\zeta(d),\zeta(c),\zeta(i),\zeta(j)\bigr\rangle=\zeta(d)
\end{equation*}
\end{cuhmodel}

As an important application, the model of accuracy described in this section can be used to reformulate the discrete model of planetary motion (Model~\ref{m:discrete}).
\begin{cuhmodel} \label{m:discrete2}
Let $S$ be the set of all quintuples $\bigl\langle\cuhinterval{r}{s},\cuhinterval{p}{q},\zeta(i),\zeta(j),\zeta(k)\bigr\rangle$ such that
\begin{align*}
m&=\Bigl\lfloor\frac{k+i}{10}\Bigr\rfloor & n&=\Bigl\lfloor\frac{k+j}{10}\Bigr\rfloor\\
r&=\frac{m}{10}-\frac{1}{100} & p&=360\biggl(\frac{n}{10}-\frac{1}{100}-\Bigl\lfloor\frac{n}{10}-\frac{1}{100}\Bigr\rfloor\biggr)\\
s&=\frac{m+1}{10}+\frac{1}{100} & q&=360\biggl(\frac{n+1}{10}+\frac{1}{100}-\Bigl\lfloor\frac{n+1}{10}+\frac{1}{100}\Bigr\rfloor\biggr)
\end{align*}
for some integers $m$ and $n$, for some integer $k$ between $-200000$ and $200000$, and where $i=-1$ or $+1$, and $j=-1$ or $+1$.  The angular position of the Earth, represented as a range of angles measured in degrees, is given by the function
\begin{equation*}
\alpha\bigl\langle\cuhinterval{r}{s},\cuhinterval{p}{q},\zeta(i),\zeta(j),\zeta(k)\bigr\rangle=\cuhinterval{p}{q}
\end{equation*}
The time interval, measured in years, is given by the function
\begin{equation*}
\tau\bigl\langle\cuhinterval{r}{s},\cuhinterval{p}{q},\zeta(i),\zeta(j),\zeta(k)\bigr\rangle=\cuhinterval{r}{s}
\end{equation*}
\end{cuhmodel}
\noindent
Note that like Model~\ref{m:discrete}, this model is faithful.  But the faithfulness, in this case, no longer requires that the results of some measurements be uncertain.  Instead, given any particular measurement, the state of the system is uncertain.  For example, there are $40$ distinct states $s$ such that $\tau(s)=\cuhinterval{0.29}{0.41}$.

In contrast to Model~\ref{m:discrete}, consider what happens if Model~\ref{m:discrete2} is used to explain measurements taken at time $t=0.306$ years.  Associated with each state
\begin{equation*} \bigl\langle\cuhinterval{r}{s},\cuhinterval{p}{q},\zeta(i),\zeta(j),\zeta(k)\bigr\rangle
\end{equation*}
in Model~\ref{m:discrete2} is an integer $k$, intended to represent the time interval $\cuhinterval{\frac{k}{100}}{\frac{k+1}{100}}$ during which the system is in that state.  At time $t=0.306$ years, $k=30$, and the system could be in one of the following four states:
\begin{align*}
&\bigl\langle\cuhinterval{0.19}{0.31},\cuhinterval{68.4}{111.6},\zeta(-1),\zeta(-1),\zeta(30)\bigr\rangle\\ &\bigl\langle\cuhinterval{0.29}{0.41},\cuhinterval{68.4}{111.6},\zeta(+1),\zeta(-1),\zeta(30)\bigr\rangle\\ &\bigl\langle\cuhinterval{0.19}{0.31},\cuhinterval{104.4}{147.6},\zeta(-1),\zeta(+1),\zeta(30)\bigr\rangle\\ &\bigl\langle\cuhinterval{0.29}{0.41},\cuhinterval{104.4}{147.6},\zeta(+1),\zeta(+1),\zeta(30)\bigr\rangle
\end{align*}
Like Model~\ref{m:discrete}, two time measurements are possible, $\cuhinterval{0.19}{0.31}$ or $\cuhinterval{0.29}{0.41}$.  And since the actual angular position of the Earth at time $t=0.306$ years is about $112$ degrees, the measured position of the Earth is $\cuhinterval{104.4}{147.6}$ degrees at that time.  Unlike Model~\ref{m:discrete}, this position measurement is compatible with either time measurement, since
\begin{align*}
\bigl\langle\cuhinterval{0.19}{0.31},\cuhinterval{104.4}{147.6}&,\zeta(-1),\zeta(+1),\zeta(30)\bigr\rangle\\
\bigl\langle\cuhinterval{0.29}{0.41},\cuhinterval{104.4}{147.6}&,\zeta(+1),\zeta(+1),\zeta(30)\bigr\rangle
\end{align*}
are both states of Model~\ref{m:discrete2}.

\section{Isomorphism Theorems} \label{s:theorems}

Given a physical model with a set $S$ of states and a set $A=\{\alpha_1,\alpha_2,\alpha_3,\ldots\}$ of observable quantities, we write $(S,A)$ as an abbreviation for that model.\footnote{In the interest of generality, the definition of a computable physical model places no restrictions on the set $A$ except that its members must be total recursive functions.  Some authors prefer to restrict their attention to finite sets $A$.  For example, see reference~\cite{jC92}.  Other authors may prefer to restrict their attention to observable quantities computed by programs which belong to a recursively enumerable set.}
\begin{cuhdef}
Two physical models $(S,A)$ and $(T,B)$ are \emph{isomorphic}\index{physical models!isomorphic} if and only if there exist bijections $\phi:S\to T$ and $\psi:A\to B$ such that $\alpha(s)=\psi(\alpha)\bigl(\phi(s)\bigr)$ for all $s\in S$ and all $\alpha\in A$.
\end{cuhdef}
\noindent
Intuitively, isomorphic physical models can be thought of as providing identical descriptions of the same system.\footnote{Rosen\index{Rosen, Robert}~\cite{rR78} defined a weaker notion of isomorphism.\index{physical models!isomorphic}  Physical models that are isomorphic in Rosen's\index{Rosen, Robert} sense are not necessarily isomorphic in the sense described here.}

Given any particular computable physical model $(S,A)$, there are many different models which are isomorphic to $(S,A)$.  The following two theorems provide some convenient forms for the representation of computable physical models.  Let $\pi_i^n$ be the \emph{projection function} that takes a length $n$ sequence of non-negative integers and outputs the $i$th element of the sequence.  That is, $\pi_i^n\langle x_1,x_2,\ldots,x_n\rangle=x_i$ for any positive integer $i\leq n$.

\begin{cuhtheorem}
If $A$ is a finite set, then the computable physical model $(S,A)$ is isomorphic to some computable physical model whose observable quantities are all projection functions.
\end{cuhtheorem}
\begin{proof}
Given a computable physical model $(S,A)$ with $A=\{\alpha_1,\alpha_2,\ldots,\alpha_n\}$, let $(T,B)$ be the computable physical model such that
\begin{equation*}
T=\bigl\{\,\bigl\langle\alpha_1(s),\alpha_2(s),\ldots,\alpha_n(s),s\bigr\rangle\bigm| s\in S\,\bigr\}
\end{equation*}
and let $B$ be the set of projection functions $\{\pi_1^{n+1},\pi_2^{n+1},\ldots,\pi_n^{n+1}\}$.  By construction, $(T,B)$ is a computable physical model isomorphic to $(S,A)$.
\end{proof}

\begin{cuhtheorem} \label{t:intstates}
If $S$ is an infinite set, then the computable physical model $(S,A)$ is isomorphic to some computable physical model whose set of states is the set of all non-negative integers.  If $S$ has $n$ elements, then the computable physical model $(S,A)$ is isomorphic to some computable physical model whose set of states is $\{0,1,\ldots,n-1\}$.
\end{cuhtheorem}
\begin{proof}
By definition, if $(S,A)$ is a computable physical model, then $S$ is a recursive set.  It immediately follows that $S$ is recursively enumerable.  In particular, if $S$ is infinite, then let $T$ be the set of non-negative integers and there is a bijective recursive function $\psi$ from $T$ to $S$.  If $S$ has $n$ elements, then there is a bijective recursive function $\psi$ from $T=\{0,1,\ldots,n-1\}$ to $S$.  Now, given $A=\{\alpha_1,\alpha_2,\alpha_3,\ldots\}$, let $B=\{\alpha_1\circ\psi,\alpha_2\circ\psi,\alpha_3\circ\psi,\ldots\}$, where $\alpha\circ\psi$ denotes the composition of the functions $\alpha$ and $\psi$.  By construction, $(S,A)$ is isomorphic to $(T,B)$.
\end{proof}

Although Theorem~\ref{t:intstates} implies that any computable physical model $(S,A)$ is isomorphic to a computable physical model $(T,B)$ where $T$ is a set of consecutive non-negative integers beginning with zero, there is no effective procedure for constructing a program that computes the characteristic function of $T$, given a program for computing the characteristic function of $S$ when $S$ is a finite set.  That is, Theorem~\ref{t:intstates} does not hold uniformly.

\begin{cuhdef}
A \emph{non-negative\index{physical models!non-negative integer} integer physical model} is a pair $(S,A)$ where $S$ is a set of non-negative integers and where each member of $A$ is a partial function from the non-negative integers to the non-negative integers.  $S$ is the set of states of the model, and $A$ is the set of observable quantities of the model.
\end{cuhdef}

Note that if $\alpha$ is an observable quantity of a non-negative integer physical model $(S,A)$, then $\alpha$ might be undefined for some inputs.  The non-negative integer physical models form a more general class of objects than the computable physical models.  In particular, a computable physical model is a non-negative integer physical model whose set of states is a recursive set and whose observable quantities are total recursive functions.
\begin{cuhtheorem} \label{t:remodel}
A non-negative integer physical model $(S,A)$ is isomorphic to some computable physical model if $S$ is a recursively enumerable set and if each member of $A$ is a partial recursive function whose domain includes all the members of $S$.
\end{cuhtheorem}
\begin{proof}
Given a non-negative integer physical model $(S,A)$, note that the construction of the computable physical model $(T,B)$ in the proof of Theorem~\ref{t:intstates} only requires that $S$ be a recursively enumerable set and that each member of $A$ be a partial recursive function whose domain includes all the members of $S$.  Therefore, any such non-negative integer physical model $(S,A)$ is isomorphic to the computable physical model $(T,B)$.
\end{proof}

Let $\phi$ be any partial recursive function and let $S$ be the largest set of consecutive non-negative integers beginning with zero such that $\phi(s)$ is defined for each $s\in S$.  Let $A$ be the set $\{\pi_1^{m}\circ\phi,\pi_2^{m}\circ\phi,\ldots,\pi_m^{m}\circ\phi\}$ for some non-negative integer $m$.  By Theorem~\ref{t:remodel}, $(S,A)$ is isomorphic to a computable physical model.  We say that any such computable physical model is \emph{determined}\index{physical models!determined by partial recursive functions} by $\phi$.
\begin{cuhtheorem}
Every computable physical model with finitely many observable quantities is determined by some partial recursive function $\phi$.
\end{cuhtheorem}
\begin{proof}
Let $(S,A)$ be a computable physical model with $A=\{\alpha_1,\alpha_2,\ldots,\alpha_m\}$.  Let $\psi$ be the recursive function given in the proof of Theorem~\ref{t:intstates}.  If $S$ has only $n$ states, then let $\phi(i)$ be undefined for all non-negative integers $i\geq n$.  Otherwise, define
\begin{equation*}
\phi(i)=\bigl\langle\alpha_1(\psi(i)),\alpha_2(\psi(i)),\ldots,\alpha_m(\psi(i))\bigr\rangle
\end{equation*}
By construction, $(S,A)$ is determined by $\phi$.
\end{proof}

A physical model $(S,A)$ is said to be \emph{reduced}\index{physical models!reduced} if and only if for each pair of distinct states $s_1$ and $s_2$ in $S$, there exists an $\alpha\in A$ with $\alpha(s_1)\ne\alpha(s_2)$.
\begin{cuhtheorem} \label{t:reduced}
If $(S,A)$ and $(T,B)$ are isomorphic physical models and $(S,A)$ is reduced, then $(T,B)$ is also a reduced physical model.
\end{cuhtheorem}
\begin{proof}
Let $(S,A)$ and $(T,B)$ be isomorphic physical models and let $(S,A)$ be reduced.  Since $(S,A)$ and $(T,B)$ are isomorphic, there exist bijections $\phi:S\to T$ and $\psi:A\to B$ such that $\alpha(s)=\psi(\alpha)\bigl(\phi(s)\bigr)$ for all $s\in S$ and all $\alpha\in A$.  Now suppose that $t_1$ and $t_2$ are distinct states in $T$.  Because $\phi$ is a bijection, $\phi^{-1}(t_1)$ and $\phi^{-1}(t_2)$ are distinct states in $S$.  But $S$ is reduced, so there must exist an $\alpha\in A$ such that $\alpha\bigl(\phi^{-1}(t_1)\bigr)\ne\alpha\bigl(\phi^{-1}(t_2)\bigr)$.  Furthermore,
\begin{equation*}
\alpha\bigl(\phi^{-1}(t_1)\bigr)=\psi(\alpha)\bigl(\phi\bigl(\phi^{-1}(t_1)\bigr)\bigr)=\psi(\alpha)(t_1)
\end{equation*}
and
\begin{equation*}
\alpha\bigl(\phi^{-1}(t_2)\bigr)=\psi(\alpha)\bigl(\phi\bigl(\phi^{-1}(t_2)\bigr)\bigr)=\psi(\alpha)(t_2)
\end{equation*}
Hence, there exists a $\beta\in B$ such that $\beta(t_1)\ne\beta(t_2)$, namely $\beta=\psi(\alpha)$.  We may conclude that the physical model $(T,B)$ is reduced.
\end{proof}

An \emph{epimorphism}\index{physical models!epimorphic} from a physical model $(S,A)$ to a physical model $(T,B)$ is a pair of functions $(\phi,\psi)$ such that $\phi$ is a surjection from $S$ to $T$ and $\psi$ is a bijection from $A$ to $B$, where $\alpha(s)=\psi(\alpha)\bigl(\phi(s)\bigr)$ for all $s\in S$ and all $\alpha\in A$.  Two physical models $(S_1,A_1)$ and $(S_2,A_2)$ are said to be \emph{observationally\index{physical models!observationally equivalent} equivalent} if and only if there are epimorphisms from $(S_1,A_1)$ to $(T,B)$ and from $(S_2,A_2)$ to $(T,B)$, where $(T,B)$ is some reduced physical model.
\begin{cuhtheorem} \label{t:obseq}
If $(S_1,A_1)$ and $(S_2,A_2)$ are isomorphic physical models, then $(S_1,A_1)$ and $(S_2,A_2)$ are observationally equivalent.
\end{cuhtheorem}
\begin{proof}
Define an equivalence relation on $S_2$ so that $r\in S_2$ is related to $s\in S_2$ if and only if $\alpha(r)=\alpha(s)$ for all $\alpha\in A_2$.
Let $T$ be the corresponding set of equivalence classes of $S_2$.  For each $\alpha\in A_2$, define a function $\alpha^\prime$ so that if $s\in t\in T$, then $\alpha^\prime(t)=\alpha(s)$.  Let $B=\{\,\alpha^\prime\mid\alpha\in A_2\,\}$.  By construction, $(T,B)$ is a reduced physical model.  Also note that there is an epimorphism $(\phi,\psi)$ from $(S_2,A_2)$ to $(T,B)$.  Namely, $\phi$ is the function that maps each member of $S_2$ to its corresponding equivalence class in $T$, and $\psi$ is the function that maps each $\alpha\in A_2$ to $\alpha^\prime\in B$.

Now suppose that $(S_1,A_1)$ and $(S_2,A_2)$ are isomorphic.  By definition, there are bijections $\phi^\prime$ from $S_1$ to $S_2$ and $\psi^\prime$ from $A_1$ to $A_2$ such that $\alpha(s)=\psi^\prime(\alpha)\bigl(\phi^\prime(s)\bigr)$ for all $s\in S_1$ and all $\alpha\in A_1$.  Since $(\phi,\psi)$ is an epimorphism from $(S_2,A_2)$ to $(T,B)$, it immediately follows that $(\phi\circ\phi^\prime,\psi\circ\psi^\prime)$ is an epimorphism from $(S_1,A_1)$ to $(T,B)$.  We may conclude, by definition, that $(S_1,A_1)$ and $(S_2,A_2)$ are observationally equivalent.
\end{proof}

Intuitively, two physical models are observationally equivalent when they both make the same observable predictions.  For example, as was discussed in Section~\ref{s:coarse}, the model obtained by omitting the observable quantities $\alpha_2$ and $\tau_2$ from Model~\ref{m:high} is observationally equivalent to the discrete model of planetary motion (Model~\ref{m:discrete}).  Moreover, if a physical model $(S_1,A_1)$ is faithful\index{physical models!faithful}, and if $(S_1,A_1)$ is observationally equivalent to $(S_2,A_2)$, then $(S_2,A_2)$ is also faithful.\index{physical models!faithful}

It is important to note that the converse of Theorem~\ref{t:obseq} does not hold.  That is, observationally equivalent models are not necessarily isomorphic. Consider, for example, the computable physical models $\bigl(\{0,1\},\{\alpha\}\bigr)$ and $\bigl(\{0,1,2\},\{\beta\}\bigr)$ where $\alpha(s)=s$ for all $s\in\{0,1\}$ and where $\beta(s)=\lfloor s/2\rfloor$ for all $s\in\{0,1,2\}$.  These models are not isomorphic because $\{0,1\}$ and $\{0,1,2\}$ have different cardinalities.  Yet, they are observationally equivalent, since both models have a single observable quantity whose only possible values are $0$ and $1$.  Physical models that are isomorphic must not only make the same observable predictions, but must also have the same \emph{structure}.  The models $\bigl(\{0,1\},\{\alpha\}\bigr)$ and $\bigl(\{0,1,2\},\{\beta\}\bigr)$ have different structures because, assuming that the states are equally likely, they both give different answers to the question, ``What is the probability that the observable quantity has value $0$?''

\section{Oracles and Effective Topologies} \label{s:oracles}

Given any set $X$, we can impose a topology on $X$.  Let $\mathcal{B}$ be a basis for this topology.  The members of $\mathcal{B}$ are said to be \emph{basis elements}.  We say that a set $\mathcal{L}_x\subseteq\mathcal{B}$ is a \emph{local basis} for a point $x\in X$ if and only if the following two conditions hold.
\begin{enumerate}
\item For each $L\in\mathcal{L}_x$, $x$ is a member of $L$.
\item For each $B\in\mathcal{B}$ with $x\in B$, there exists an $L\in\mathcal{L}_x$ with $L\subseteq B$.
\end{enumerate}
Note that every point $x\in X$ has a local basis.  For example, the set
\begin{equation*}
\mathcal{L}_x=\{\,B\in\mathcal{B}\mid x\in B\,\}
\end{equation*}
of all basis elements that contain $x$ is a local basis for $x$.

If the basis $\mathcal{B}$ is countable, then each basis element can be encoded as a non-negative integer.  In that case, choose some encoding and let $\nu(n)$ be the basis element encoded by $n$.  We allow for the possibility that a basis element may be encoded by more than one non-negative integer.  (That is, $\nu$ is not necessarily an injection.)  The \emph{domain} of $\nu$, denoted $\cuhdom_\mathcal{B}\nu$, is the set of all non-negative integers $n$ such that $\nu(n)\in\mathcal{B}$.  For any function $\phi:A\to B$ and any set $C\subseteq A$, let $\phi(C)=\{\phi(c)\mid c\in C\}$ denote the \emph{image} of $C$ under $\phi$.  We let $\mathbb{N}$ denote the set of non-negative integers.
\begin{cuhdef}
A function $\phi:\mathbb{N}\to\cuhdom_\mathcal{B}\nu$ is said to be an \emph{oracle}\index{oracles for real numbers} for a point $x$, with basis $\mathcal{B}$ and coding $\nu$, if and only if $\nu\bigl(\phi(\mathbb{N})\bigr)$ is a local basis for $x$.
\end{cuhdef}
\noindent
An oracle $\phi$ for $x$ is \emph{complete}\index{oracles for real numbers!complete} if and only if every $n\in\cuhdom_\mathcal{B}\nu$ such that $x\in\nu(n)$ is a member of $\phi(\mathbb{N})$.  An oracle is said to be \emph{nested}\index{oracles for real numbers!nested} if and only if $\nu\bigl(\phi(n+1)\bigr)\subseteq\nu\bigl(\phi(n)\bigr)$ for all $n\in\mathbb{N}$.

A pair $(\mathcal{B},\nu)$ is said to be an \emph{effective\index{topologies!effective} topology} if and only if $\mathcal{B}$ is a countable basis for a $T_0$ topology and $\nu$ is a coding for $\mathcal{B}$.  Effective topologies were first introduced in the theory of type-two\index{type-two effectivity} effectivity~\cite{KW85b,WG09}.  In accordance with that theory, we use an oracle for $x$, with a basis $\mathcal{B}$ and coding $\nu$, as a representation of the point $x$ in an effective topology $(\mathcal{B},\nu)$.  Because effective topologies are $T_0$, no two distinct points are ever represented by the same oracle.

Of special interest are effective topologies where the subset relation
\begin{equation*}
\{\,\langle b_1,b_2\rangle\mid\nu(b_1)\subseteq\nu(b_2)\;\&\;b_1\in\cuhdom_\mathcal{B}\nu\;\&\;b_2\in\cuhdom_\mathcal{B}\nu\,\}
\end{equation*}
is a recursively enumerable set.\footnote{An effective topology with a recursively enumerable subset relation is an example of a \emph{computable\index{topologies!computable} topology}, as defined in reference~\cite{WG09}.  Not all computable topologies have recursively enumerable subset relations.}  In particular, if $(\mathcal{B},\nu)$ has a recursively enumerable subset relation, then
\begin{equation*}
\cuhdom_\mathcal{B}\nu=\{\,b\mid\nu(b)\subseteq\nu(b)\;\&\;b\in\cuhdom_\mathcal{B}\nu\,\}
\end{equation*}
is also a recursively enumerable set.
\begin{cuhtheorem} \label{t:complete}
Let $\phi$ be an oracle for $x$ in an effective topology $(\mathcal{B},\nu)$ with a recursively enumerable subset relation.  Then there exists a complete oracle $\psi$ for $x$ in $(\mathcal{B},\nu)$ that is recursive relative to $\phi$ uniformly.
\end{cuhtheorem}
\begin{proof}
Suppose that $\phi$ is an oracle for $x$ in an effective topology $(\mathcal{B},\nu)$ with a recursively enumerable subset relation.  Since $\nu\bigl(\phi(\mathbb{N})\bigr)$ is a local basis for $x$, it follows that for each $b\in\cuhdom_\mathcal{B}\nu$, $x\in\nu(b)$ if and only if there exists an $n\in\mathbb{N}$ such that $\nu\bigr(\phi(n)\bigl)\subseteq\nu(b)$.  Hence,
\begin{equation*}
\bigl\{\,b\in\cuhdom_\mathcal{B}\nu\bigm|(\exists\,n\in\mathbb{N})\bigl[\nu\bigl(\phi(n)\bigr)\subseteq\nu(b)\bigr]\,\bigr\}
\end{equation*}
is the set of encodings of all basis elements that contain $x$.  But this set is recursively enumerable relative to $\phi$ because $(\mathcal{B},\nu)$ has a recursively enumerable subset relation.  Therefore, there exists a function $\psi:\mathbb{N}\to\cuhdom_\mathcal{B}\nu$, recursive relative to $\phi$, such that $\nu\bigl(\psi(\mathbb{N})\bigr)$ is this set.  By definition, $\psi$ is a complete oracle for $x$.
\end{proof}

\begin{cuhtheorem} \label{t:nest}
Let $\phi$ be an oracle for $x$ in an effective topology $(\mathcal{B},\nu)$ with a recursively enumerable subset relation.  Then there exists a nested oracle $\psi$ for $x$ in $(\mathcal{B},\nu)$ that is recursive relative to $\phi$ uniformly.
\end{cuhtheorem}
\begin{proof}
Suppose that $\phi$ is an oracle for $x$ in an effective topology $(\mathcal{B},\nu)$ with a recursively enumerable subset relation.  Note that for each pair of basis elements $B_1$ and $B_2$ such that $x\in B_1\cap B_2$, there exists a basis element $B_3$ with $x\in B_3\subseteq B_1\cap B_2$, by the definition of a basis.  Therefore, since $\nu\bigl(\phi(\mathbb{N})\bigr)$ is a local basis for $x$, there must exist an $m\in\mathbb{N}$ such that
\begin{equation*}
x\in\nu\bigl(\phi(m)\bigr)\subseteq B_3\subseteq B_1\cap B_2
\end{equation*}
Now define $\psi:\mathbb{N}\to\cuhdom_\mathcal{B}\nu$ recursively, relative to $\phi$, as follows.  Let $\psi(0)=\phi(0)$ and for each $n\in\mathbb{N}$ let $\psi(n+1)=\phi(m)$ for some $m\in\mathbb{N}$ such that
\begin{equation*}
\nu\bigl(\phi(m)\bigr)\subseteq\nu\bigl(\psi(n)\bigr)\cap\nu\bigl(\phi(n+1)\bigr)
\end{equation*}
We can find $m$ recursively given $\psi(n)$ and $\phi$ because the subset relation for $(\mathcal{B},\nu)$ is recursively enumerable, and the set of all $m\in\mathbb{N}$ such that
\begin{equation*}
\nu\bigl(\phi(m)\bigr)\subseteq\nu\bigl(\psi(n)\bigr)\;\&\;\nu\bigl(\phi(m)\bigr)\subseteq\nu\bigl(\phi(n+1)\bigr)
\end{equation*}
is therefore recursively enumerable relative to $\phi$.  We may conclude that $\psi$ is nested because
\begin{equation*}
\nu\bigl(\psi(n+1)\bigr)\subseteq\nu\bigl(\psi(n)\bigr)\cap\nu\bigl(\phi(n+1)\bigr)\subseteq\nu\bigl(\psi(n)\bigr)
\end{equation*}
for all $n\in\mathbb{N}$, and that $\psi$ is an oracle for $x$ because $x\in\nu\bigl(\psi(0)\bigr)=\nu\bigl(\phi(0)\bigr)$ and
\begin{equation*}
x\in\nu\bigl(\psi(n+1)\bigr)\subseteq\nu\bigl(\psi(n)\bigr)\cap\nu\bigl(\phi(n+1)\bigr)\subseteq\nu\bigl(\phi(n+1)\bigr)
\end{equation*}
for all $n\in\mathbb{N}$.
\end{proof}

Define $\iota\cuhinterval{a}{b}$ to be the set of all real numbers $x$ such that $a<x<b$, and let $\mathcal{I}$ be the set of all $\iota\cuhinterval{a}{b}$ such that $a$ and $b$ are rational numbers with $a<b$.  The members of $\mathcal{I}$ are said to be \emph{rational intervals}.  Note that $\mathcal{I}$ is a basis for the standard topology of the real numbers.  Indeed, the oracles for real numbers that were introduced in Section~\ref{s:continuous} were nested oracles with basis $\mathcal{I}$ and coding $\iota$.  Another basis for the standard topology of the real numbers is the set $\mathcal{I}_{10,c}$ of \emph{decimal intervals} with accuracy factor $c$, where $c$ is a positive rational number, and where $\mathcal{I}_{10,c}$ is defined to be the set of all $\iota\cuhinterval{a}{b}$ such that
\begin{align*}
a&=\frac{m}{10^n}-\frac{c}{10^n} & b&=\frac{m+1}{10^n}+\frac{c}{10^n}
\end{align*}
for some integer $m$ and some positive integer $n$.  We call $n$ the \emph{number of digits of precision} of $\cuhinterval{a}{b}$.

Note that both $(\mathcal{I},\iota)$ and $(\mathcal{I}_{10,c}\,,\iota)$ have recursively enumerable subset relations.  The following theorem asserts that if $\phi$ is an oracle for a real number $x$ with basis $\mathcal{I}$ and coding $\iota$, then there exists an oracle $\psi$ for $x$ with basis $\mathcal{I}_{10,c}$ and coding $\iota$ that is recursive relative to $\phi$ uniformly.
\begin{cuhtheorem} \label{t:convert}
Let $(\mathcal{A},\nu)$ and $(\mathcal{B},\nu)$ be effective topologies with recursively enumerable subset relations such that $\mathcal{B}\subseteq\mathcal{A}$, and such that $\mathcal{A}$ and $\mathcal{B}$ are bases for the same topology.  If $\phi$ is an oracle for $x$ in $(\mathcal{A},\nu)$, then there exists an oracle $\psi$ for $x$ in $(\mathcal{B},\nu)$ that is recursive relative to $\phi$ uniformly.
\end{cuhtheorem}
\begin{proof}
Suppose that $(\mathcal{A},\nu)$ and $(\mathcal{B},\nu)$ are effective topologies as described in the statement of the theorem, and that $\phi$ is an oracle for a point $x$ in $(\mathcal{A},\nu)$.  By definition,
\begin{equation*}
x\in\nu\bigl(\phi(n)\bigr)
\end{equation*}
for every $n\in\mathbb{N}$.  And because $\nu\bigl(\phi(n)\bigr)$ is an open set, it is a union of basis elements from $\mathcal{B}$.  Hence, there must exist a $B\in\mathcal{B}$ such that
\begin{equation*}
x\in B\subseteq\nu\bigl(\phi(n)\bigr)
\end{equation*}
But $\nu\bigl(\phi(\mathbb{N})\bigr)$ is a local basis for $x$, and $B$ is a basis element in $\mathcal{A}$, so there exists an $m\in\mathbb{N}$ such that $x\in\nu\bigl(\phi(m)\bigr)\subseteq B$.  Therefore, we have that for each $n\in\mathbb{N}$ there exist $B\in\mathcal{B}$ and $m\in\mathbb{N}$ such that
\begin{equation*}
x\in\nu\bigl(\phi(m)\bigr)\subseteq B\subseteq\nu\bigl(\phi(n)\bigr)
\end{equation*}
Now, since $(\mathcal{B},\nu)$ has a recursively enumerable subset relation, $\cuhdom_\mathcal{B}\nu$ is a recursively enumerable set.  Then, because $(\mathcal{A},\nu)$ also has a recursively enumerable subset relation, the set
\begin{equation*}
\bigl\{\,\langle m,b\rangle \bigm|\nu\bigl(\phi(m)\bigr)\subseteq\nu(b)\subseteq\nu\bigl(\phi(n)\bigr)\;\&\;m\in\mathbb{N}\;\&\;b\in\cuhdom_\mathcal{B}\nu\,\bigr\}
\end{equation*}
is recursively enumerable relative to $\phi$, for any $n\in\mathbb{N}$.  Therefore, there is a function $\psi:\mathbb{N}\to\cuhdom_\mathcal{B}\nu$, recursive relative to $\phi$, such that $\psi(n)=b$ for all $n\in\mathbb{N}$, where
\begin{equation*}
\nu\bigl(\phi(m)\bigr)\subseteq\nu(b)\subseteq\nu\bigl(\phi(n)\bigr)
\end{equation*}
for some $m\in\mathbb{N}$.  But this function $\psi$ is an oracle for $x$ in $(\mathcal{B},\nu)$, because
\begin{equation*}
x\in\nu\bigl(\phi(m)\bigr)\subseteq\nu\bigl(\psi(n)\bigr)\subseteq\nu\bigl(\phi(n)\bigr)
\end{equation*}
for all $n\in\mathbb{N}$.
\end{proof}

\section{Basic Representations of Sets}

\begin{cuhdef}
Let $(\mathcal{B},\nu)$ be an effective topology on a set $X$, and let $A$ be any subset of $X$.  We say that a set $R$ of non-negative integers is a \emph{basic\index{basic representations} representation} of $A$ in the effective topology $(\mathcal{B},\nu)$ if and only if the following two conditions hold.
\begin{enumerate}
\item  $R\subseteq\cuhdom_\mathcal{B}\nu$
\item \label{i:basic} $x\in A$ if and only if there exists a local basis $\mathcal{L}_x$ for $x$ with $\mathcal{L}_x\subseteq\nu(R)$.
\end{enumerate}
\end{cuhdef}
\noindent
Note that condition~\ref{i:basic} of the definition ensures that no two distinct sets in $(\mathcal{B},\nu)$ have the same basic representation.  Note further that if $R$ is a basic representation of $A$, then $\{\,A\cap\nu(r)\mid r\in R\,\}$ is a basis for the subspace topology on $A$, and this is an effective topology with coding $\lambda r\bigl[A\cap\nu(r)\bigr]$.

In an effective topology we use basic representations to represent sets of points, but not all sets of points have basic representations.  For example, there are $2^{2^{\aleph_0}}$ many sets of real numbers, but since a basic representation is a set of non-negative integers, there are at most $2^{\aleph_0}$ many basic representations.  Nevertheless, many commonly-studied sets have basic representations.\footnote{In the effective topology $(\mathcal{I},\iota)$, the set of rational numbers does not have a basic representation, but the set of irrational numbers has the basic representation
\begin{equation*}
R=\Bigl\{\,\cuhBinterval{\frac{m}{n!}}{\frac{m+1}{n!}}\Bigm|m\in\mathbb{Z}\;\&\;n\in\mathbb{N}\,\Bigr\}
\end{equation*}
where $\mathbb{Z}$ denotes the set of integers.  It is tempting to conjecture that the sets with basic representations in an effective topology $(\mathcal{B},\nu)$ are exactly the $G_\delta$ sets, but there is a trivial counterexample to this conjecture if the effective topology is not $T_1$.}
\begin{cuhtheorem}
Let $A$ be a set in an effective topology $(\mathcal{B},\nu)$.
\begin{enumerate}
\item If $A$ is an open set, then $A$ has a basic representation.
\item If $A$ is a closed set, then $A$ has a basic representation.
\end{enumerate}
\end{cuhtheorem}
\begin{proof}
Suppose that $A$ is an open set in the effective topology $(\mathcal{B},\nu)$ and let
\begin{equation*}
R=\{\,r\in\cuhdom_\mathcal{B}\nu\mid\nu(r)\subseteq A\,\}
\end{equation*}
Clearly, $R\subseteq\cuhdom_\mathcal{B}\nu$ and if $x\notin A$ then there does not exist a local basis $\mathcal{L}_x$ for $x$ with $\mathcal{L}_x\subseteq\nu(R)$, since no member of $\nu(R)$ contains $x$.  Alternatively, if $x\in A$ then, by the definition of a basis, for each basis element $B_1$ that contains $x$ there exists some basis element $B_2$ such that $x\in B_2\subseteq B_1\cap A$.  That is, if $x\in A$ then for each basis element $B_1$ with $x\in B_1$, there exists a basis element $B_2\in\{\,B\in\mathcal{B}\mid x\in B\subseteq A\,\}$ with $B_2\subseteq B_1$.  It immediately follows that $\mathcal{L}_x=\{\,B\in\mathcal{B}\mid x\in B\subseteq A\,\}$ is a local basis for $x$ and $\mathcal{L}_x\subseteq \nu(R)$.  By definition, $R$ is a basic representation of $A$.

Now, if $A$ is a closed set in $(\mathcal{B},\nu)$, then let
\begin{equation*}
R=\{\,r\in\cuhdom_\mathcal{B}\nu\mid A\cap\nu(r)\ne\varnothing\,\}
\end{equation*}
Clearly, $R\subseteq\cuhdom_\mathcal{B}\nu$ and if $x\in A$ then there exists a local basis $\mathcal{L}_x$ for $x$ with $\mathcal{L}_x\subseteq\nu(R)$.  Namely, $\mathcal{L}_x$ is the set of all basis elements that contain $x$.  Alternatively, if $x\notin A$, then since $A$ is closed, every local basis $\mathcal{L}_x$ for $x$ contains a basis element that does not intersect $A$.  Therefore,  $\mathcal{L}_x\nsubseteq\nu(R)$.   We may conclude, by definition, that $R$ is a basic representation of $A$.
\end{proof}

Although we use a basic representation $R$ to represent a set of points in an effective topology, the following theorem demonstrates that there is, in general, no effective procedure (relative to $R$) for finding oracles for those points.  Nevertheless, if we restrict our attention to certain special classes of basic representations $R$, then effective procedures do exist.  See Section~\ref{s:kreisel}.
\begin{cuhtheorem} \label{t:broracle}
Let $\mathcal{B}$ be a countable basis for the standard topology of $\mathbb{R}^n$ and let $\nu$ be a coding for the basis.  Then there does not exist a partial recursive function $\phi$ satisfying the condition that for every singleton set $\{x\}\subseteq\mathbb{R}^n$ and for every basic representation $R_x$ of $\{x\}$ in the effective topology $(\mathcal{B},\nu)$, the function $\lambda m\bigl[\phi(R_x,m)\bigr]$ is an oracle for $x$ in $(\mathcal{B},\nu)$.
\end{cuhtheorem}
\begin{proof}
Let $(\mathcal{B},\nu)$ be an effective topology as in the statement of the theorem and suppose, as an assumption to be shown contradictory, that there exists a partial recursive function $\phi$ satisfying the condition that for every singleton set $\{x\}\subseteq\mathbb{R}^n$ and for every basic representation $R_x$ of $\{x\}$ in the effective topology $(\mathcal{B},\nu)$, the function $\lambda m\bigl[\phi(R_x,m)\bigr]$ is an oracle for $x$ in $(\mathcal{B},\nu)$.  Now consider any two distinct points $x\in\mathbb{R}^n$ and $y\in\mathbb{R}^n$, and let $R_x$ be a basic representation of $\{x\}$ in $(\mathcal{B},\nu)$.  Since the standard topology of $\mathbb{R}^n$ is $T_1$, there must exist a non-negative integer $k$ such that $\phi(R_x,k)$ is defined and
\begin{equation*}
y\notin\nu\bigl(\phi(R_x,k)\bigr)
\end{equation*}
Next, choose a program for computing $\phi$.  Note that since the computation for $\phi(R_x,k)$ has only finitely many steps, only finitely many non-negative integers are tested for membership in $R_x$ during the course of the computation.  Let $C$ be the collection of all $i\in\mathbb{N}$ such that $i\in R_x$ and such that $i$ is tested for membership in $R_x$ during the course of the computation of $\phi(R_x,k)$.  Similarly, let $D$ be the collection of all $i\in\mathbb{N}$ such that $i\notin R_x$ and such that $i$ is tested for membership in $R_x$ during the course of the computation of $\phi(R_x,k)$.

Now, choose any oracle $\psi$ for $y$ in $(\mathcal{B},\nu)$.  Note that $\psi(\mathbb{N})$ is a basic representation for $\{y\}$ in $(\mathcal{B},\nu)$.  And because $C$ and $D$ are finite sets,
\begin{equation*}
R_y=\bigl(\psi(\mathbb{N})\cup C\bigr)-D
\end{equation*}
is also a basic representation for $\{y\}$ in $(\mathcal{B},\nu)$.  It follows that $\phi(R_x,k)=\phi(R_y,k)$, because whenever $i$ is tested for membership in $R_x$ during the course of the computation of $\phi(R_x,k)$, $i\in R_x$ if and only if $i\in R_y$.  Therefore,
\begin{equation*}
y\notin\nu\bigl(\phi(R_x,k)\bigr)=\nu\bigl(\phi(R_y,k)\bigr)
\end{equation*}
But by the definition of $\phi$, $\lambda m\bigl[\phi(R_y,m)\bigr]$ is an oracle for $y$.  Hence,
\begin{equation*}
y\in\nu\bigl(\phi(R_y,k)\bigr)
\end{equation*}
This is a contradiction, so the assumption must be false.  The partial recursive function $\phi$ does not exist.
\end{proof}

\section{Basic Representations of Physical Models}

Let $\mathbb{R}$ be the set of all real numbers.  For any two sets $A$ and $B$, let $A\times B=\{\,(a,b)\mid a\in A\;\&\;b\in B\,\}$ be the Cartesian product of $A$ with $B$.  We write $A^k$ to denote the set formed by taking the Cartesian product of $A$ with itself $k$ many times.  For example, $A^3=(A\times A)\times A$.  As with Cantor's pairing function, $(a,b,c)$ is an abbreviation for $((a,b),c)$, and so on.  Similarly, we define the \emph{Cartesian projection function} $\varpi_i^n$ so that $\varpi_i^n(x_1,x_2,\ldots,x_n)=x_i$ for each positive integer $i\leq n$.

A physical model $(S,A)$ with finitely many observable quantities is said to be in \emph{normal form}\index{physical models!normal forms of} if and only if $S\subseteq\mathbb{R}^n$ and $A=\{\varpi_1^n,\varpi_2^n,\ldots,\varpi_n^n\}$.
\begin{cuhtheorem}
The following two conditions hold for any physical model $(S,A)$ with finitely many observable quantities.
\begin{enumerate}
\item \label{i:obseq} $(S,A)$ is observationally equivalent to a physical model in normal form.
\item \label{i:isomorph} $(S,A)$ is isomorphic to a physical model in normal form if and only if $(S,A)$ is a reduced physical model.
\end{enumerate}
\end{cuhtheorem}
\begin{proof}
Begin by noting that if $(T,B)$ is a physical model in normal form, and if $t_1\ne t_2$ for any $t_1\in T$ and $t_2\in T$, then $\varpi_i^n(t_1)\ne\varpi_i^n(t_2)$ for some positive integer $i\leq n$.  Therefore, by definition, every physical model in normal form is a reduced physical model.  It immediately follows from Theorem~\ref{t:reduced} that if a physical model $(S,A)$ is isomorphic to a physical model in normal form, then $(S,A)$ is a reduced physical model.

To prove condition~\ref{i:obseq}, suppose that $(S,A)$ is a physical model such that $A=\{\alpha_1,\alpha_2,\ldots,\alpha_n\}$.  Define
\begin{equation*}
T=\bigl\{\,\bigl(\alpha_1(s),\alpha_2(s),\ldots,\alpha_n(s)\bigr)\bigm|s\in S\,\bigr\}
\end{equation*}
and let $B=\{\varpi_1^n,\varpi_2^n,\ldots,\varpi_n^n\}$.  Note that $(T,B)$ is a physical model in normal form.  Also note that the function $\phi:S\to T$ given by
\begin{equation*}
\phi(s)=\bigl(\alpha_1(s),\alpha_2(s),\ldots,\alpha_n(s)\bigr)
\end{equation*}
is a surjection, and that $\alpha_i(s)=\varpi_i^n\bigl(\phi(s)\bigr)$ for all $s\in S$ and all positive integers $i\leq n$.  Therefore, there is an epimorphism from $(S,A)$ to the reduced physical model $(T,B)$.  Trivially, there is also an epimorphism from $(T,B)$ to itself.  We may conclude that $(S,A)$ is observationally equivalent to $(T,B)$.

To prove condition~\ref{i:isomorph}, consider the special case where $(S,A)$ is a reduced physical model.  Because $(S,A)$ is reduced, we have that if $s_1\ne s_2$ for any $s_1\in S$ and $s_2\in S$, then there exists a positive integer $i\leq n$ such that $\alpha_i(s_1)\ne\alpha_i(s_2)$.  This implies that if $s_1\ne s_2$ then $\phi(s_1)\ne\phi(s_2)$.  Hence, $\phi$ is an injection.  Since $\phi$ is also a surjection, $\phi$ is a bijection.  Therefore, if $(S,A)$ is a reduced physical model, then $(S,A)$ and $(T,B)$ are isomorphic.  We have already proved the converse, that if $(S,A)$ is isomorphic to a physical model in normal form, then $(S,A)$ is a reduced physical model.  Hence, condition~\ref{i:isomorph} holds.
\end{proof}

The notion of a basic representation of a set can be generalized so that we may speak of basic representations of physical models in normal form.  Given a physical model $(S,A)$ in normal form with $A=\{\varpi_1^n,\varpi_2^n,\ldots,\varpi_n^n\}$, we may choose sets $X_1$, $X_2$,~$\ldots$~,~$X_n$ such that $\varpi_i^n(S)\subseteq X_i\subseteq\mathbb{R}$ for each positive integer $i\leq n$, and we may impose effective topologies $(\mathcal{B}_1,\nu_1)$, $(\mathcal{B}_2,\nu_2)$,~$\ldots$~,~$(\mathcal{B}_n,\nu_n)$ on these sets.\footnote{If $(S,A)$ is faithful,\index{physical models!faithful} then the bases for these topologies are uniquely determined by the physical operations used to measure each of the observable quantities.  For example, if an observable quantity is an angle measurement, then the corresponding topology is the topology of a circle, and each basis element corresponds to a particular reading on the instrument that is used to measure angles.  The idea that basis elements correspond to the values of measurements appears to have originated with reference~\cite{WZ02}.}  Define
\begin{equation*}
(\mathcal{B}_1,\nu_1)\otimes(\mathcal{B}_2,\nu_2)\otimes\cdots\otimes(\mathcal{B}_n,\nu_n)
\end{equation*}
to be the effective topology with basis $\mathcal{B}$ such that
\begin{equation*}
\mathcal{B}=\{\,B_1\times B_2\times\cdots\times B_n\mid B_1\in\mathcal{B}_1\;\&\;B_2\in\mathcal{B}_2\;\&\;\cdots\;\&\;B_n\in\mathcal{B}_n\,\}
\end{equation*}
and with coding $\nu$ such that
\begin{equation*}
\nu\langle a_1,a_2,\ldots,a_n\rangle=\nu_1(a_1)\times\nu_2(a_2)\times\cdots\times\nu_n(a_n)
\end{equation*}
We call $(\mathcal{B},\nu)$ the \emph{effective\index{effective product} product} of $(\mathcal{B}_1,\nu_1)$, $(\mathcal{B}_2,\nu_2)$,~$\ldots$~,~$(\mathcal{B}_n,\nu_n)$.  Note that $S$ is a set of points in the effective topology $(\mathcal{B},\nu)$.  A physical model $(R,H)$ is said to be a \emph{basic\index{basic representations}\index{physical models!basic representations of} representation} of the physical model $(S,A)$ if $R$ is a basic representation of $S$ in the effective topology $(\mathcal{B},\nu)$ and if $H$ is the set $\{\pi_1^n,\pi_2^n,\ldots,\pi_n^n\}$ of projection functions.

For example, the non-discrete continuous computable physical model of planetary motion (Model~\ref{m:continuous}) is a basic representation of the simple model of planetary motion (Model~\ref{m:real}).  In particular, Model~\ref{m:continuous} is obtained by imposing the effective topology $(\mathcal{I}_{10,c}\,,\iota)$ on the time in Model~\ref{m:real}, where $c=\frac{1}{10}$, and by imposing the effective topology described in Footnote~\ref{n:angle} on the angular position in Model~\ref{m:real}.  The product of these topologies is the topology for the surface of a cylinder.  The states of Model~\ref{m:real} are a spiral path on the surface of that cylinder, and the set of states of Model~\ref{m:continuous} is a basic representation of the path.

\section{Data and Predictions}

In order to make predictions, we are often interested in finding the set of all states of a physical model which could account for a given collection of simultaneous measurements.  That is, given a physical model $(S,A)$ with $A=\{\alpha_1,\alpha_2,\alpha_3,\ldots\}$, and given real numbers $x_1$, $x_2$,~$\ldots$~,~$x_k$, we are interested in the set
\begin{equation*}
P=\{\,s\in S\mid\alpha_1(s)=x_1\;\&\;\alpha_2(s)=x_2\;\&\;\cdots\;\&\;\alpha_k(s)=x_k\,\}
\end{equation*}
In this context, the real numbers $x_1$, $x_2$,~$\ldots$~,~$x_k$ are said to be the \emph{data}, and $P$ is the corresponding set of states \emph{predicted} by the model.

The following theorem shows that if we are given a basic representation of a physical model $(S,A)$ in normal form, together with complete oracles for the real numbers $x_1$, $x_2$,~$\ldots$~,~$x_k$, then there is an effective procedure for finding a basic representation of the set $P$, provided that the underlying topology is $T_1$.  (This is a rather weak requirement, since almost all topologies with practical applications in the sciences are $T_1$.)
\begin{cuhtheorem} \label{t:brpredict}
Let $(S,A)$ be a physical model in normal form with $A=\{\varpi_1^n,\varpi_2^n,\linebreak[0]\ldots,\varpi_n^n\}$ and let $(R,H)$ be a basic representation of $(S,A)$ in a $T_1$ effective topology
\begin{equation*}
(\mathcal{B},\nu)=(\mathcal{B}_1,\nu_1)\otimes(\mathcal{B}_2,\nu_2)\otimes\cdots\otimes(\mathcal{B}_n,\nu_n)
\end{equation*}
If $k\leq n$ and $\phi_1$, $\phi_2$,~$\ldots$~,~$\phi_k$ are complete oracles for $x_1$, $x_2$,~$\ldots$~,~$x_k$ in the effective topologies $(\mathcal{B}_1,\nu_1)$, $(\mathcal{B}_2,\nu_2)$,~$\ldots$~,~$(\mathcal{B}_k,\nu_k)$, then there is a basic representation of
\begin{equation*}
P=\{\,s\in S\mid\varpi_1^n(s)=x_1\;\&\;\varpi_2^n(s)=x_2\;\&\;\cdots\;\&\;\varpi_k^n(s)=x_k\,\}
\end{equation*}
in $(\mathcal{B},\nu)$ that is recursively enumerable relative to $R$, $\phi_1$, $\phi_2$,~$\ldots$~,~$\phi_k$ uniformly.
\end{cuhtheorem}
\begin{proof}
Let the variables be defined as in the statement of the theorem and note that the set
\begin{equation*}
Q=\bigl\{\,r\in R\bigm|\bigl(\forall\,i\in\{1,2,\ldots,k\}\bigr)\bigl(\exists\,m\in\mathbb{N}\bigr)\bigl[\pi_i^n(r)=\phi_i(m)\bigr]\,\bigr\}
\end{equation*}
is recursively enumerable relative to $R$, $\phi_1$, $\phi_2$,~$\ldots$~,~$\phi_k$ uniformly.  (In fact, $Q$ is recursively enumerable relative to $\xi$, $\phi_1$, $\phi_2$,~$\ldots$~,~$\phi_k$ uniformly, where $\xi$ is a function that merely enumerates the members of $R$.)  We claim that $Q$ is a basic representation of $P$ in $(\mathcal{B},\nu)$.  Since $Q\subseteq R\subseteq\cuhdom_\mathcal{B}\nu$,
it suffices to prove that $s\in P$ if and only if there exists a local basis $\mathcal{L}_s$ for $s$ with $\mathcal{L}_s\subseteq\nu(Q)$.
Or equivalently, it suffices to prove that $s\in P$ if and only if there exists an oracle $\psi$ for $s$ with $\psi(\mathbb{N})\subseteq Q$.

Suppose $s\in P$.  Because $R$ is a basic representation of $S$ in $(\mathcal{B},\nu)$, there is an oracle $\psi$ for $s$ in $(\mathcal{B},\nu)$ such that $\psi(\mathbb{N})\subseteq R$.  Moreover, for each positive integer $i\leq k$, the set $\nu_i\bigl(\pi_i^n\bigl(\psi(\mathbb{N})\bigr)\bigr)$ is a local basis for $\varpi_i^n(s)=x_i$.  And since $\phi_i$ is a complete oracle for $x_i$, we have that $\pi_i^n\bigl(\psi(\mathbb{N})\bigr)\subseteq\phi_i(\mathbb{N})$.  Hence, if $r=\psi(l)$ for some $l\in\mathbb{N}$, then there exists an $m\in\mathbb{N}$ such that $\pi_i^n(r)=\phi_i(m)$.  Therefore, by the definition of $Q$, $\psi(\mathbb{N})\subseteq Q$.

Conversely, suppose that $\psi$ is an oracle for some point $s$ in $(\mathcal{B},\nu)$, and that $\psi(\mathbb{N})\subseteq Q$.  Then, for each positive integer $i\leq k$, we have that $\pi_i^n\bigl(\psi(\mathbb{N})\bigr)\subseteq\phi_i(\mathbb{N})$.  Of course, $\nu_i\bigl(\pi_i^n\bigl(\psi(\mathbb{N})\bigr)\bigr)$ is a local basis for $\varpi_i^n(s)$ because $\nu\bigl(\psi(\mathbb{N})\bigr)$ is a local basis for $s$.  And by the definition of $\phi_i$, $\nu_i\bigl(\phi_i(\mathbb{N})\bigr)$ is a local basis for $x_i$.  Hence, a local basis for $\varpi_i^n(s)$ is a subset of a local basis for $x_i$ in the effective topology $(\mathcal{B}_i,\nu_i)$.  But because $(\mathcal{B},\nu)$ is a $T_1$ effective topology, $(\mathcal{B}_i,\nu_i)$ is also $T_1$.  In a $T_1$ topology, local bases for any two distinct points $z_1$ and $z_2$ must contain basis elements $B_1$ and $B_2$, respectively, such that $z_2\notin B_1$ and $z_1\notin B_2$.  Therefore, since a local basis for $\varpi_i^n(s)$ is a subset of a local basis for $x_i$, it must be the case that $\varpi_i^n(s)=x_i$.  We may conclude, by the definition of $P$, that $s\in P$.
\end{proof}

A set $S\subseteq\mathbb{R}^n$ is said to be the \emph{graph} of a function $\psi:\mathbb{R}^k\to\mathbb{R}^{n-k}$, if and only if
\begin{equation*}
S=\{\,(x_1,\ldots,x_k,x_{k+1},\ldots,x_n)\in\mathbb{R}^n\mid\psi(x_1,\ldots,x_k)=(x_{k+1},\ldots,x_n)\,\}
\end{equation*}
And we say that a physical model $(S,A)$ is \emph{induced}\index{physical models!induced by a function} by a function $\psi:\mathbb{R}^k\to\mathbb{R}^{n-k}$ if and only if $(S,A)$ is in normal form and $S$ is the graph of $\psi$.  Therefore, if $(S,A)$ is induced by $\psi:\mathbb{R}^k\to\mathbb{R}^{n-k}$ and we are given $k$ real numbers $x_1$, $x_2$,~$\ldots$~,~$x_k$ as data, then the corresponding set of states predicted by $(S,A)$ is a singleton set $P=\{s\}$.  Namely,
\begin{equation*}
s=(x_1,x_2,\ldots,x_k,x_{k+1},\ldots,x_n)
\end{equation*}
where $x_{k+1}$, $x_{k+2}$,~$\ldots$~,~$x_n$ are the real numbers uniquely determined by the equation
\begin{equation*}
\psi(x_1,x_2,\ldots,x_k)=(x_{k+1},x_{k+2},\ldots,x_n)
\end{equation*}
It then follows from Theorem~\ref{t:brpredict} that given complete oracles for $x_1$, $x_2$,~$\ldots$~,~$x_k$ in the standard topology of $\mathbb{R}$, and given a basic representation $R$ of $(S,A)$ in the standard topology of $\mathbb{R}^n$, there is an effective procedure (relative to the given oracles and $R$) for finding a basic representation of $\{s\}$.  But by Theorem~\ref{t:broracle} there is, in general, no effective procedure for finding an \emph{oracle} for $s$.  In the next section we describe a special class of basic representations for which such an effective procedure does exist.

\section{Kreisel's Criterion} \label{s:kreisel}

A common way to interpret Kreisel's\index{Kreisel's criterion} criterion is to say that a physical model $(S,A)$ satisfies Kreisel's criterion on $\mathbb{R}^k$ if and only if $(S,A)$ is induced by a function $\psi:\mathbb{R}^k\to\mathbb{R}^{n-k}$ for some positive integer $n>k$, and for each positive integer $j\leq n-k$ there is a partial recursive function $\kappa_j$ such that if $\phi_1$, $\phi_2$,~$\ldots$~,~$\phi_k$ are nested oracles for real numbers $x_1$, $x_2$,~$\ldots$~,~$x_k$ in the effective topology $(\mathcal{I},\iota)$, then $\kappa_j(\phi_1,\phi_2,\ldots,\phi_k,m)$ is defined for all $m\in\mathbb{N}$ and $\lambda m\bigl[\kappa_j(\phi_1,\phi_2,\ldots,\phi_k,m)\bigr]$ is a nested oracle for $\varpi_j^{n-k}\bigl(\psi(x_1,x_2,\ldots,x_k)\bigr)$ in $(\mathcal{I},\iota)$.  Note by Theorem~\ref{t:convert} that the effective topology $(\mathcal{I},\iota)$ in this statement can be replaced, without loss of generality, with any effective topology $(\mathcal{B},\iota)$ that has a recursively enumerable subset relation and such that $\mathcal{B}\subseteq\mathcal{I}$ is a basis for the standard topology of the real numbers.

Practical computer models that use multiple-precision interval\index{interval arithmetic} arithmetic~\cite{rM66} provide examples of physical models satisfying Kreisel's criterion.  Typically, such models are induced by a function $\psi:\mathbb{R}^k\to\mathbb{R}^{n-k}$ where, for each positive integer $j\leq n-k$, there is a recursive function $\xi_j$ such that if the data $x_1$, $x_2$,~$\ldots$~,~$x_k$ lie within the intervals $\cuhinterval{a_1}{b_1}$, $\cuhinterval{a_2}{b_2}$,~$\ldots$~,~$\cuhinterval{a_k}{b_k}$ respectively, then $\varpi_j^{n-k}\bigl(\psi(x_1,x_2,\linebreak[0]\ldots,x_k)\bigr)$ lies within the interval
\begin{equation*}
\xi_j\bigl(\cuhinterval{a_1}{b_1},\cuhinterval{a_2}{b_2},\ldots,\cuhinterval{a_k}{b_k}\bigr)
\end{equation*}
In such a case, the partial recursive function $\kappa_j$ in Kreisel's criterion is given by
\begin{equation*}
\kappa_j(\phi_1,\phi_2,\ldots,\phi_k,m)=\xi_j\bigl(\phi_1(m),\phi_2(m),\ldots,\phi_k(m)\bigr)
\end{equation*}

We are now prepared to state the following theorem, which holds uniformly.
\begin{cuhtheorem} \label{t:kreisel}
If a physical model satisfies Kreisel's criterion on $\mathbb{R}^k$, then the model has a basic representation that is isomorphic to a computable physical model.
\end{cuhtheorem}
\begin{proof}
Suppose that $(S,A)$ is a physical model satisfying Kreisel's criterion on $\mathbb{R}^k$.  In particular, suppose that $(S,A)$ is induced by a function $\psi:\mathbb{R}^k\to\mathbb{R}^{n-k}$, and for each positive integer $j\leq n-k$ suppose there is a partial recursive function $\kappa_j$ such that if $\phi_1$, $\phi_2$,~$\ldots$~,~$\phi_k$ are nested oracles for real numbers $x_1$, $x_2$,~$\ldots$~,~$x_k$ in the effective topology $(\mathcal{I}_{10,c}\,,\iota)$, then $\kappa_j(\phi_1,\phi_2,\ldots,\phi_k,m)$ is defined for all $m\in\mathbb{N}$ and $\lambda m\bigl[\kappa_j(\phi_1,\phi_2,\ldots,\phi_k,m)\bigr]$ is a nested oracle for $\varpi_j^{n-k}\bigl(\psi(x_1,x_2,\ldots,x_k)\bigr)$ in $(\mathcal{I}_{10,c}\,,\iota)$, where $c$ is a positive rational number.

Let $I=\cuhdom_{\mathcal{I}_{10,c}}\iota$.  Then, for each interval $u\in I$ with midpoint $p$ and with $d$ digits of precision, define the partial recursive function $\sigma_u$ so that
\begin{equation*}
\sigma_u(l)=
\begin{cases}
\cuhomicron_p(l) &\text{if $l<d$}\\
\text{undefined} &\text{if $l\geq d$}
\end{cases}
\end{equation*}
for each $l\in\mathbb{N}$, where $\cuhomicron_p$ is the standard decimal oracle described in Section~\ref{s:continuous}.  Note that for each $x\in\mathbb{R}$ and each $m\in\mathbb{N}$, $\sigma_{\cuhomicron_x(m)}(l)=\cuhomicron_x(l)$ for all non-negative integers $l\leq m$.  Now, for each positive integer $j\leq n-k$, choose a program to compute $\kappa_j$ and let $\kappa_j(\sigma_{u_1},\sigma_{u_2},\ldots,\sigma_{u_k},m)$ be undefined if for some positive integer $i\leq k$ and some $l\in\mathbb{N}$ the program calls $\sigma_{u_i}(l)$ in the course of the computation and $\sigma_{u_i}(l)$ is undefined.  Note that given $u_1$, $u_2$,~$\ldots$~,~$u_k$, the set of all $m\in\mathbb{N}$ such that $\kappa_j(\sigma_{u_1},\sigma_{u_2},\ldots,\sigma_{u_k},m)$ is defined is a recursively enumerable set, since for each $m$ we can follow the computation and test whether or not $\sigma_{u_i}(l)$ is defined whenever $\sigma_{u_i}(l)$ is called by the program, for any $i$ and $l$.  Let $R$ be the set of all $\langle u_1,u_2,\ldots,u_n\rangle$ such that $u_i\in I$ for each positive integer $i\leq k$, and such that $u_{k+j}=\kappa_j(\sigma_{u_1},\sigma_{u_2},\ldots,\sigma_{u_k},m)$ for some $m\in\mathbb{N}$ if $j\leq n-k$ is a positive integer.  Note that $R$ is also recursively enumerable.  Let $H=\{\pi_1^n,\pi_2^n,\ldots,\pi_n^n\}$.  We claim that $(R,H)$ is a basic representation of $(S,A)$.

As a brief digression from the proof, suppose that $\langle u_1,u_2,\ldots,u_n\rangle\in R$ and note that for each positive integer $i\leq k$, if $x_i\in\iota(u_i)$ then there exists a nested oracle $\phi_i$ for $x_i$ in $(\mathcal{I}_{10,c}\,,\iota)$ such that $\phi_i(m)=\sigma_{u_i}(m)$ for all $m\in\mathbb{N}$ where $\sigma_{u_i}(m)$ is defined.  And since for each positive integer $j\leq n-k$ we have that $\lambda m\bigl[\kappa_j(\phi_1,\phi_2,\ldots,\phi_k,m)\bigr]$ is a nested oracle for $x_{k+j}=\varpi_j^{n-k}\bigl(\psi(x_1,x_2,\ldots,x_k)\bigr)$, it follows that
\begin{equation*}
x_{k+j}\in\iota\bigl(\kappa_j(\phi_1,\phi_2,\ldots,\phi_k,m)\bigr)=\iota\bigl(\kappa_j(\sigma_{u_1},\sigma_{u_2},\ldots,\sigma_{u_k},m)\bigr)=\iota(u_{k+j})
\end{equation*}
for some $m\in\mathbb{N}$.  Hence, given any $\langle u_1,u_2,\ldots,u_n\rangle\in R$, if $x_i\in\iota(u_i)$ for each positive integer $i\leq k$, then $x_{k+j}=\varpi_j^{n-k}\bigl(\psi(x_1,x_2,\ldots,x_k)\bigr)\in\iota(u_{k+j})$ for each positive integer $j\leq n-k$.

Now, returning to the proof of Theorem~\ref{t:kreisel}, let
\begin{equation*}
(\mathcal{B},\nu)=\overbrace{(\mathcal{I}_{10,c}\,,\iota)\otimes(\mathcal{I}_{10,c}\,,\iota)\otimes\cdots\otimes(\mathcal{I}_{10,c}\,,\iota)}^{\text{$n$ factors}}
\end{equation*}
and note that the basis of $(\mathcal{B},\nu)$ is a basis for the standard topology of $\mathbb{R}^n$.  Note that because $(S,A)$ satisfies Kreisel's criterion, the function $\psi$ which induces $(S,A)$ is continuous.  And since continuous real functions have closed graphs, the set $S$ is closed in the standard topology of $\mathbb{R}^n$.  Because $S$ is closed, to prove that $R$ is a basic representation of $S$ in $(\mathcal{B},\nu)$, it suffices to show that for each $r\in R$ there exists an $x\in S$ with $x\in\nu(r)$, and that for each $x\in S$ there is a local basis $\mathcal{L}_x$ for $x$ with $\mathcal{L}_x\subseteq\nu(R)$.

By the definition of $R$, if $r=\langle u_1,u_2,\ldots,u_n\rangle\in R$ then for each positive integer $j\leq n-k$,
\begin{equation*}
u_{k+j}=\kappa_j(\sigma_{u_1},\sigma_{u_2},\ldots,\sigma_{u_k},m)
\end{equation*}
for some $m\in\mathbb{N}$.  So, if $p_i$ is the midpoint of the interval $u_i$ for each positive integer $i\leq k$ , then
\begin{equation*}
u_{k+j}=\kappa_j(\cuhomicron_{p_1},\cuhomicron_{p_2},\ldots,\cuhomicron_{p_k},m)
\end{equation*}
And by Kreisel's criterion $\lambda m\bigl[\kappa_j(\cuhomicron_{p_1},\cuhomicron_{p_2},\ldots,\cuhomicron_{p_k},m)\bigr]$ is an oracle for $\varpi_j^{n-k}\bigl(\psi(p_1,\linebreak[0]p_2,\ldots,p_k)\bigr)$.  Therefore, for each $r\in R$ there exists an $(x_1,x_2,\ldots,x_n)\in S$ with $(x_1,x_2,\ldots,x_n)\in\nu(r)$.  Namely, $x_i=p_i$ for each positive integer $i\leq k$ and
\begin{equation*}
x_{k+j}=\varpi_j^{n-k}\bigl(\psi(p_1,p_2,\ldots,p_k)\bigr)
\end{equation*}
for each positive integer $j\leq n-k$.

Now suppose that $(x_1,x_2,\ldots,x_n)$ is an arbitrary member of $S$.  Again, by Kreisel's criterion, for each positive integer $j\leq n-k$, the function $\lambda m\bigl[\kappa_j(\cuhomicron_{x_1},\cuhomicron_{x_2},\linebreak[0]\ldots,\cuhomicron_{x_k},m)\bigr]$ is an oracle for $x_{k+j}$.  But for each $m\in\mathbb{N}$ and each positive integer $i\leq k$, the computation for $\kappa_j(\cuhomicron_{x_1},\cuhomicron_{x_2},\ldots,\cuhomicron_{x_k},m)$ has only finitely many steps, and so the oracle $\cuhomicron_{x_i}$ can only be called finitely many times during the course of the computation.  Hence, for each $m\in\mathbb{N}$ there exists a non-negative integer $l_i$ for each $i\leq k$, such that for any non-negative integer $l_i^\prime\geq l_i$, if $u_i=\cuhomicron_{x_i}(l_i^\prime)$ then $\kappa_j(\sigma_{u_1},\sigma_{u_2},\ldots,\sigma_{u_k},m)$ is defined and
\begin{equation*}
\kappa_j(\sigma_{u_1},\sigma_{u_2},\ldots,\sigma_{u_k},m)=\kappa_j(\cuhomicron_{x_1},\cuhomicron_{x_2},\ldots,\cuhomicron_{x_k},m)
\end{equation*}
Of course, for each positive integer $j\leq n-k$ the interval
\begin{equation*} u_{k+j}=\kappa_j(\cuhomicron_{x_1},\cuhomicron_{x_2},\ldots,\cuhomicron_{x_k},m)
\end{equation*}
can be made arbitrarily small by choosing a suitably large value of $m$, and for each positive integer $i\leq k$ the interval $u_i$ can be made arbitrarily small by choosing a suitably large value of $l_i^\prime$.  Furthermore, by definition, $\langle u_1,u_2,\ldots,u_n\rangle\in R$.  It immediately follows that for each $x=(x_1,x_2,\ldots,x_n)\in S$ there is a local basis $\mathcal{L}_x$ for $x$ such that $\mathcal{L}_x\subseteq\nu(R)$.  We may conclude that $(R,H)$ is a basic representation of $(S,A)$.  And since $R$ is recursively enumerable, it follows from Theorem~\ref{t:remodel} that $(R,H)$ is isomorphic to a computable physical model.
\end{proof}

A physical model $(S,A)$ that satisfies Kreisel's criterion on $\mathbb{R}^k$ is uniquely determined by the functions $\kappa_1$, $\kappa_2$,~$\ldots$~,~$\kappa_{n-k}$.  Moreover, the proof of Theorem~\ref{t:kreisel} describes an effective procedure for finding a basic representation of $(S,A)$, given programs for computing $\kappa_1$, $\kappa_2$,~$\ldots$~,~$\kappa_{n-k}$.  Let $\mathcal{K}_{k,n,c}$ be the collection of all basic representations of physical models that are constructed from physical models satisfying Kreisel's criterion according to the procedure in the proof of Theorem~\ref{t:kreisel}, where $c$ is the positive rational number which appears in that proof.  An immediate question is whether there exists an effective procedure for the inverse operation.  That is, given a basic representation in $\mathcal{K}_{k,n,c}$, is there an effective procedure for constructing partial recursive functions $\kappa_1$, $\kappa_2$,~$\ldots$~,~$\kappa_{n-k}$?  In the proof of the following theorem, we show that the answer is ``Yes.''  Therefore, for every physical model satisfying Kreisel's criterion on $\mathbb{R}^k$, there is a computable physical model that may be used in its place, to predict the values of observable quantities given the data.
\begin{cuhtheorem}
If $(R,H)\in\mathcal{K}_{k,n,c}$ and if $\psi:\mathbb{R}^k\to\mathbb{R}^{n-k}$ is the function whose graph has basic representation $R$, then there exist partial recursive functions $\kappa_1$, $\kappa_2$,~$\ldots$~,~$\kappa_{n-k}$ such that if $\phi_1$, $\phi_2$,~$\ldots$~,~$\phi_k$ are nested oracles for real numbers $x_1$, $x_2$,~$\ldots$~,~$x_k$ in the effective topology $(\mathcal{I}_{10,c}\,,\iota)$, then for each positive integer $j\leq n-k$, $\kappa_j(\phi_1,\phi_2,\ldots,\phi_k,m)$ is defined for all $m\in\mathbb{N}$ and $\lambda m\bigl[\kappa_j(\phi_1,\phi_2,\ldots,\phi_k,m)\bigr]$ is a nested oracle for $\varpi_j^{n-k}\bigl(\psi(x_1,x_2,\ldots,x_k)\bigr)$ in $(\mathcal{I}_{10,c}\,,\iota)$.
\end{cuhtheorem}
\begin{proof}
Suppose that $(R,H)\in\mathcal{K}_{k,n,c}$ and that $\psi:\mathbb{R}^k\to\mathbb{R}^{n-k}$ is the function whose graph has basic representation $R$.  Note by the proof of Theorem~\ref{t:kreisel} that $R$ is recursively enumerable.  Now, given any oracles $\phi_1$, $\phi_2$,~$\ldots$~,~$\phi_k$ for real numbers $x_1$, $x_2$,~$\ldots$~,~$x_k$ in $(\mathcal{I}_{10,c}\,,\iota)$, it follows from Theorem~\ref{t:complete} that there are complete oracles $\phi_1^\prime$, $\phi_2^\prime$,~$\ldots$~,~$\phi_k^\prime$ for $x_1$, $x_2$,~$\ldots$~,~$x_k$ in $(\mathcal{I}_{10,c}\,,\iota)$, such that $\phi_1^\prime$, $\phi_2^\prime$,~$\ldots$~,~$\phi_k^\prime$ are recursive relative to $\phi_1$, $\phi_2$,~$\ldots$~,~$\phi_k$ uniformly.  Let $x_{k+1}$, $x_{k+2}$,~$\ldots$~,~$x_n$ be the real numbers uniquely determined by the equation
\begin{equation*}
\psi(x_1,x_2,\ldots,x_k)=(x_{k+1},x_{k+2},\ldots,x_n)
\end{equation*}
Then by the proof of Theorem~\ref{t:brpredict},
\begin{equation*}
Q=\bigl\{\,\langle u_1,u_2,\ldots,u_n\rangle\in R\bigm|\bigl(\forall\,i\in\{1,2,\ldots,k\}\bigr)\bigl(\exists\,m\in\mathbb{N}\bigr)\bigl[u_i=\phi_i^\prime(m)\bigr]\,\bigr\}
\end{equation*}
is a basic representation of $\bigl\{(x_1,x_2,\ldots,x_n)\bigr\}$ in
\begin{equation*}
(\mathcal{B},\nu)=\overbrace{(\mathcal{I}_{10,c}\,,\iota)\otimes(\mathcal{I}_{10,c}\,,\iota)\otimes\cdots\otimes(\mathcal{I}_{10,c}\,,\iota)}^{\text{$n$ factors}}
\end{equation*}
And since $R$ is recursively enumerable, the set $Q$ is recursively enumerable relative $\phi_1$, $\phi_2$,~$\ldots$~,~$\phi_k$ uniformly.

Now, since $x_i\in\iota\bigl(\phi_i(m)\bigr)$ for each positive integer $i\leq k$, it follows from the definition of $Q$ that $x_i\in\iota(u_i)$ for each $\langle u_1,u_2,\ldots,u_n\rangle\in Q$.  But recall from the proof of Theorem~\ref{t:kreisel} that $R$ has the property that if $x_i\in\iota(u_i)$ for each positive integer $i\leq k$, then $x_j\in\iota(u_{k+j})$ for each positive integer $j\leq n-k$.  Hence,
\begin{equation*}
(x_1,x_2,\ldots,x_n)\in\nu\langle u_1,u_2,\ldots,u_n\rangle
\end{equation*}
for each $\langle u_1,u_2,\ldots,u_n\rangle\in Q$.  It immediately follows from the definition of a basic representation that $\nu(Q)$ is a local basis for the point $(x_1,x_2,\ldots,x_n)$.  Therefore, for any function $\kappa:\mathbb{N}\to R$ such that $\kappa(\mathbb{N})=Q$, the function $\kappa$ is an oracle for $(x_1,x_2,\ldots,x_n)$ in $(\mathcal{B},\nu)$.  And since $Q$ is recursively enumerable relative to $\phi_1$, $\phi_2$,~$\ldots$~,~$\phi_k$ uniformly, there is a nested oracle $\kappa$ that is recursive relative to $\phi_1$, $\phi_2$,~$\ldots$~,~$\phi_k$ uniformly.  So, if we define
\begin{equation*}
\kappa_j(\phi_1,\phi_2,\ldots,\phi_k,m)=\pi_{k+j}^n\bigl(\kappa(m)\bigr)
\end{equation*}
for each positive integer $j\leq n-k$ and for each $m\in\mathbb{N}$, then $\kappa_j$ is partial recursive, $\kappa_j(\phi_1,\phi_2,\ldots,\linebreak[0]\phi_k,m)$ is defined for all $m\in\mathbb{N}$, and $\lambda m\bigl[\kappa_j(\phi_1,\phi_2,\ldots,\phi_k,m)\bigr]$ is a nested oracle for $\varpi_j^{n-k}\bigl(\psi(x_1,x_2,\ldots,x_k)\bigr)=x_{k+j}$ in $(\mathcal{I}_{10,c}\,,\iota)$.
\end{proof}

\section{Acknowledgments}

This paper is adapted from a chapter of our Ph.D. thesis~\cite{mS10}.  We received many helpful suggestions from our thesis advisor, Richard\index{Statman, Richard} Statman, and from the members of our thesis committee, most notably Robert\index{Batterman, Robert} Batterman and Lenore\index{Blum, Lenore} Blum.  In the same regard, we also benefited from the careful reading and subsequent suggestions of Kevin\index{Kelly, Kevin} Kelly, Klaus\index{Weihrauch, Klaus} Weihrauch, Stephen\index{Wolfram, Stephen} Wolfram, and Hector\index{Zenil, Hector} Zenil.  Zenil\index{Zenil, Hector} has also made us aware of a draft paper, posted by Marcus\index{Hutter, Marcus} Hutter~\cite{mH09}, which expresses ideas very similar to those presented here.

\end{document}